\documentclass[]{siamart220329}

\usepackage{lipsum}
\usepackage{amsfonts}
\usepackage{graphicx}
\usepackage{epstopdf}
\ifpdf
  \DeclareGraphicsExtensions{.eps,.pdf,.png,.jpg}
\else
  \DeclareGraphicsExtensions{.eps}
\fi
\usepackage[useregional]{datetime2}

\newsiamremark{remark}{Remark}
\newsiamremark{hypothesis}{Hypothesis}
\crefname{hypothesis}{Hypothesis}{Hypotheses}
\newsiamthm{claim}{Claim}

\headers{Counter Examples}{Varner and Patel}

\title{The Challenges of Optimization for Data Science 
\thanks
{
  Submitted to the editors \today.
  \funding{}
}
}

\author{
Christian Varner\thanks{Department of Statistics, University of Wisconsin Madison, Madison, WI 
  (\email{cvarner@wisc.edu}, \email{vivak.patel@wisc.edu}).}
\and 
Vivak Patel\footnotemark[2]}

\usepackage{amsopn}

\usepackage[utf8]{inputenc}
\usepackage{xcolor}
\usepackage{mathtools}
\usepackage{enumitem}
\usepackage{amsmath}

{
\theoremstyle{plain}

\newtheorem{property}{Property}
}

\usepackage{natbib}
\setcitestyle{numbers}

\usepackage[para,online,flushleft]{threeparttable}
\usepackage{amsfonts}
\usepackage{amssymb}
\usepackage{float}
\usepackage{algpseudocode}
\usepackage{subfiles}
\usepackage{indentfirst}
\usepackage{graphicx}
\usepackage{tabularx}
\usepackage{caption}
\usepackage{array,booktabs}
\usepackage{subcaption}
\usepackage{ragged2e}
\usepackage{tcolorbox}
\tcbuselibrary{breakable}
\tcbset{
  width=0.99\textwidth,
  halign=justify,
  center,
  breakable,
  colback=white    
}
\usepackage{tikz}
\usetikzlibrary{plotmarks}
\usetikzlibrary{patterns}
\usepackage{pgfplotstable}
\usepackage{pgfplots}
\pgfplotsset{compat=1.15}
\usepgfplotslibrary{statistics}
\pgfplotsset{boxplot legend/.style={
    legend image code/.code={
        \draw[#1] (0cm,0cm) rectangle (0.6cm,0.3cm)
        (0.3cm,0cm) -- (0.3cm,-0.1cm) (0.1cm,-0.1cm) -- (0.5cm,-0.1cm)
        (0.3cm,0.3cm) -- (0.3cm,0.4cm) (0.1cm,0.4cm) -- (0.5cm,0.4cm);
    },
}}

\usepackage{pifont}

\usepackage{caption}
\usepackage[export]{adjustbox}
\usepackage[margin=1in]{geometry}
\usepackage{todonotes}

\makeatletter
\newenvironment{breakablealgorithm}
  {
   \begin{center}
     \refstepcounter{algorithm}
     \hrule height.8pt depth0pt \kern2pt
     \renewcommand{\caption}[2][\relax]{
       {\raggedright\textbf{\fname@algorithm~\thealgorithm} ##2\par}%
       \ifx\relax##1\relax 
         \addcontentsline{loa}{algorithm}{\protect\numberline{\thealgorithm}##2}%
       \else 
         \addcontentsline{loa}{algorithm}{\protect\numberline{\thealgorithm}##1}%
       \fi
       \kern2pt\hrule\kern2pt
     }
  }{
     \kern2pt\hrule\relax
   \end{center}
  }
\makeatother

\crefname{assumption}{assumption}{assumptions}
\crefname{property}{property}{properties}
\crefname{definition}{definition}{definitions}
\crefformat{footnote}{#2\footnotemark[#1]#3}

\newcommand{\defeq}{\vcentcolon=}

\newcommand{\inlinenorm}[1]{ \Vert #1 \Vert}
\newcommand{\norm}[1]{\left\Vert #1 \right\Vert}

\newcommand{\bigpar}[1]{\left( #1 \right)}

\begin{document}

\maketitle

\begin{abstract}
Optimization problems arising in data science have given rise to a number of new derivative-based optimization methods. 
Such methods often use standard smoothness assumptions---namely, global Lipschitz continuity of the gradient function---to establish a convergence theory. 
Unfortunately, in this work, we show that common optimization problems from data science applications are not globally Lipschitz smooth, nor do they satisfy some more recently developed smoothness conditions in literature. 
Instead, we show that such optimization problems are better modeled as having locally Lipschitz continuous gradients. 
We then construct explicit examples satisfying this assumption on which existing classes of optimization methods are either unreliable or experience an explosion in evaluation complexity. 
In summary, we show that optimization problems arising in data science are particularly difficult to solve, and that there is a need for methods that can reliably and practically solve these problems.
\end{abstract}



\section{Introduction} \label{sec:introduction}
Optimization problems arising in data science have resulted in a surge of new optimization methods that are designed for the particular computational characteristics of such problems. 
These optimization methods also come with impressive computational complexity results, often claiming optimal complexity for a particular class of optimization problems \cite[for an introduction, see][\S 1.2.5, \S 1.2.6, \S 2.5]{cartis2022}. 
While these results are undeniable, they typically require optimization problems to have gradient functions that are globally Lipschitz continuous.\footnote{While we will be precise later, a globally Lipschitz continuous function is a function that is Lipschitz continuous with common rank over any compact set.} 
Unfortunately, global Lipschitz continuity of the gradient function seems to be inapplicable for some of the most common data science problems; for example, feed forward networks, recurrent neural networks, and certain types of regression do not always satisfy this condition \cite[see][Propositions 1 to 3]{patel2022globalconvergence}.
Furthermore, losing global Lipschitz continuity of the gradient function can have serious consequences on the reliability of optimization methods. For example, gradient descent with diminishing step sizes will globally converge for objective functions with globally Lipschitz continuous gradients \cite[Proposition 1.2.4]{bertsekas1999}; but, for objective functions with a \emph{locally} Lipschitz continuous gradient function, this method can produce iterates with a diverging optimality gap and gradient values that are uniformly bounded away from zero \cite[Proposition 4.4]{patel2023gradient}.

This motivates the need to accurately model the continuity conditions for optimization problems arising in data science. While local Lipschitz continuity of the gradient function seems reasonable \cite[see][Propositions 1 to 3]{patel2022globalconvergence}, other measures of continuity of the gradient function---falling between global Lipschitz continuity of the gradient function and local Lipschitz continuity of the gradient function---have recently appeared in the literature \cite{zhang2020gradientclipping,patel2022globalconvergence,li2023convex}. Such alternative measures of continuity seem to yield better convergence behavior for some optimization methods \cite{zhang2020gradientclipping,patel2022globalconvergence,li2023convex}; however, it is not understood if these intermediate conditions are realistic for optimization problems arising in data science, and whether there are consequences for optimization algorithms when these intermediate conditions are not satisfied. 

\begin{table}[ht!]
\centering
\caption{Below are canonical data science problems, a reference for the problem, and different smoothness conditions: globally Lipschitz smooth (GRA); globally Lipschitz Hessian smooth (HES); subquadratically Lipschitz smooth (SUB); quadratically Lipschitz smooth (QUA); and local Lipschitz continuity of the gradient (LOC). We indicate ``yes'' if the problem satisfies the condition and ``no'' if not. Details can be found in \cref{subsec:smoothness-overview}.}
\label{table-data-science-problems}
\footnotesize
\begin{tabular}{@{}>{\raggedright}p{2in}lllllll@{}} \toprule
\textbf{Problem} & \textbf{Ref.} & \textbf{GRA} & \textbf{HES} & \textbf{SUB} & \textbf{QUA} & \textbf{LOC} & \textbf{Details} \\ \midrule
Factor Analysis & \cite[\S 4]{lawley1962factor} & No & No & No & Yes & Yes & \Cref{subsec:factor-analysis}\\
Feed Forward Neural Network & \cite[Chapter 8]{rumelhart1987} & No & No & No & Yes & Yes & \Cref{subsec:classification-nn}\\
Generalized Estimating Equations & \cite[Chapter 3]{lipsitz2008geelongitudinal} & No & No & No & Yes & Yes & \Cref{subsec:gee}\\
Inverse Gaussian Regression & \cite[Chapter 8]{chhikara1989} & No & No & No & No & Yes & \Cref{subsec:inv-gaussian-reg}
\\\bottomrule
\end{tabular}
\end{table}

We address these open questions by first analyzing the applicability of key intermediate measures of continuity for common data science problems (see \Cref{table-data-science-problems}).
We conclude that these intermediate measures of continuity of the gradient function \textit{do not} always apply to optimization problems arising in data science (details are in \cref{sec:smoothness}). 
In light of this, we consider the possible problematic behavior of optimization methods under the \textit{applicable} assumption of local Lipschitz continuity of the gradient function. 
By working with two broad categorizations of derivative-based optimization methods, we construct two building block functions (see \cref{subsec:cd-preliminaries,subsec:exploding-obj-eval-preliminaries})---improving on instances in previous works \cite{vavasis1993,cartis2022,patel2023gradient}---that allow us to generate explicit examples for which one category of optimization methods will diverge (see \cref{table-objective-free}), while the other will incur an exponential growth in objective function evaluations between accepted iterates (see \cref{table-multiple-objective-evaluation}).
In other words, we show, under assumptions appropriate for optimization problems arising in data science, existing methods face serious practical challenges. 

\begin{table}[ht]
\centering
\caption{Gradient methods that do not use objective function information. The method, convergence result reference(s), and a reference to a divergence example are listed.}
\label{table-objective-free}
\footnotesize
\begin{tabular}{@{}>{\raggedright}p{2in}>{\raggedright}p{3in}l@{}} \toprule
\textbf{Method} & \textbf{Convergence} & \textbf{Divergence} \\ \midrule
Diminishing Step-Size & \cite[Proposition 1.2.4]{bertsekas1999}, \cite[Theorem 3.7]{patel2023gradient} & \cite[Proposition 4.4]{patel2023gradient} \\
Constant Step-Size & \cite[Corollary 1]{armijo1966minimization}, \cite[Theorem 3]{zhang2020gradientclipping}, \cite[Theorem 5.1]{li2023convex} & \Cref{subsec:ce-constant} \\
Barzilai-Borwein Methods & \cite[\S 4]{barzilai1988}, \cite[Theorem 3.3]{burdakov2019stabilized} & \Cref{subsec:ce-bbb} \\
Nesterov's Acceleration Method & \cite[Theorem 6]{nesterov2012GradientMF}, \cite[Theorem 4.4]{li2023convex} & \Cref{subsec:ce-nag} \\
Bregman Distance Method & \cite[Theorem 1]{bauschke2017descentlemma}	&  \Cref{subsec:ce-bregman}\\
Negative Curvature Method & \cite[Theorem 1]{curtis2018exploiting} & \Cref{subsec:ce-negative}\\
Lipschitz Approximation & \cite[Theorem 1]{malitsky2020adaptive} & \Cref{subsec:ce-lipschitz}\\
Weighted Gradient-Norm Damping & \cite[Theorem 2.3]{wu2020wngrad}, \cite[Corollary 1]{grapiglia2022AdaptiveTrust} & \Cref{subsec:ce-weighted}\\
Adaptively Scaled Trust Region & \cite[Corollary 3.3]{gratton2022firstorderOFFO}, \cite[Theorem 3.10]{gratton2022firstorderOFFO}, \cite[Theorem 3.3]{gratton2023OFFOSecondOrderOpt} & \Cref{subsec:ce-adaptiveTR}\\ \bottomrule
\end{tabular}
\end{table}

Owing to this work, we have rigorously specified appropriate assumptions for smooth optimization problems arising in data science (which have not been considered in the established literature), and have shown how to construct general purpose examples satisfying these assumptions for which optimization methods can experience impractical behavior (which can be used to analyze emerging methods). Importantly, we hope that this work motivates a deeper exploration of the properties of optimization problems arising in data science, and encourages new methods that are reliable and practical for optimization problems in data science \cite[e.g.,][]{varner2023novelgradient}.

\begin{table}[ht]
\centering
\caption{Gradient methods that may use multiple objective function evaluations per accepted iterate. For each method, a short description, a convergence result reference, and an exploding evaluation example are given.}
\label{table-multiple-objective-evaluation}
\footnotesize
\begin{tabular}{@{}>{\raggedright}p{2.75in}>{\raggedright}p{2in}l@{}} \toprule
\textbf{Method} & \textbf{Convergence} & \textbf{Exploding Evaluations} \\ \midrule
Armijo's Backtracking Method & \cite[Corollary 2]{armijo1966minimization}, \cite[Theorem 5]{zhang2020firstorder} & \Cref{subsec:armijo-backtracking}\\
Newton's Method with Cubic Regularization & \cite[Theorem 1]{nesterov2006cubic} & \Cref{subsec:cubic-newton-method}\\
Adaptive Cubic Regularization &  \cite[Theorem 2.5]{cartis2011cubic1}, \cite[Theorem 2.1]{cartis2011cubic2} & \Cref{subsec:adaptive-cubic-reg}\\
Lipschitz Constant(s) Line Search Methods & \cite[Theorem 3]{nesterov2012GradientMF}, \cite[Theorem 2]{curtis2018exploiting} & \Cref{subsec:lipschitz-line-search}\\
\bottomrule
\end{tabular}
\end{table}

\section{Optimization Problems for Common Data Science Applications} \label{sec:smoothness}
Optimization problems arising in data science are commonly formulated as
\begin{equation}
\min_{\theta} F(\theta),
\end{equation}
for an objective function $F:\mathbb{R}^n \to \mathbb{R}$, such that $F(\theta)$ is bounded from below, and its gradient, $\dot F(\theta) \defeq \nabla F(\psi) |_{\psi = \theta}$, exists everywhere and is continuous. In addition, when applicable, we let $\ddot F(\theta) = \nabla^2 F(\psi) |_{\psi=\theta}$ and $\dddot F(\theta) = \nabla^3 F(\psi) |_{\psi=\theta}$. In all examples where $\dddot F(\theta)$ is used, it is a scalar quantity. 
In what follows, we explore the continuity properties of $F(\theta)$'s that arise in data science applications. We begin by establishing definitions about continuity.

\subsection{Notions of Continuity} \label{subsec:smoothness-overview}
Below, we present several definitions of continuity of $\dot F$ and $\ddot F$ that appear in literature. As these notions of continuity are still being developed, there is not one standard presentation, but we try to stay consistent with related work. While we can present most of these definitions without assuming the existence and continuity of $\ddot F$, we will do so to uniformize the discussion. All norms that follow are arbitrary, and when needed, we can restrict these definitions to subsets of $\mathbb{R}^n$. We now provide the various definitions of continuity.

\begin{definition} \label{def-grad-globally-lip-cont}
    $\dot F$ is globally Lipschitz continuous if $\exists \mathcal{L}_G \geq 0$ such that $\forall \psi \in \mathbb{R}^n, ~\inlinenorm{ \ddot F(\psi)} \leq \mathcal{L}_G$. We say that $F$ is globally Lipschitz smooth.
\end{definition}

\begin{definition} \label{def-hes-globally-lip-cont}
    $\ddot F$ is globally Lipschitz continuous if $\exists \mathcal{L}_H \geq 0$ such that $\forall \psi, \phi \in \mathbb{R}^n$, $\inlinenorm{ \ddot F(\psi) - \ddot F(\phi) } \leq \mathcal{L}_H\inlinenorm{ \psi - \phi}$. We say that $F$ is globally Lipschitz Hessian smooth.
\end{definition}
\begin{remark}
    If $F(\theta)$ is third order continuously differentiable, an equivalent definition of globally Lipschitz Hessian smooth is if $\forall\psi\in\mathbb{R}^n, ~\inlinenorm{\dddot F(\psi)} \leq \mathcal{L}_H$ \cite[see][Theorem A.8.5]{cartis2022}.
\end{remark}

\begin{definition} \label{def-grad-subquadratic-lip-cont}
    $\dot F$ is subquadratically Lipschitz continuous if 
    \begin{equation}
    \limsup_{k\to\infty} \inlinenorm{ \ddot F(\theta_k)} / \inlinenorm{ \dot F(\theta_k)}^2 = 0
    \end{equation}
    for any sequence, $\lbrace \theta_k:k \in \mathbb{N} \rbrace$, such that $\lim_{k \to \infty} \inlinenorm{ \dot F(\theta_k)} = \infty$. We say that $F$ is subquadratically Lipschitz smooth. See \cite{li2023convex}.
\end{definition}

\begin{definition} \label{def-grad-quadratic-lip-cont}
    $\dot F$ is quadratically Lipschitz continuous if 
    \begin{remunerate}
        \item[(A)] for any sequence, $\lbrace \theta_k:k \in \mathbb{N} \rbrace$, such that $\lim_{k \to \infty} \inlinenorm{ \dot F(\theta_k)} = \infty$
        \begin{equation}
        \limsup_{k\to\infty} \inlinenorm{ \ddot F(\theta_k)} / \inlinenorm{ \dot F(\theta_k) }^2 < \infty;
        \end{equation}
        \item[(B)] there exists one such sequence such that $\liminf_{k\to\infty} \inlinenorm{ \ddot F(\theta_k)} / \inlinenorm{ \dot F(\theta_k) }^2 > 0$.
    \end{remunerate}
    We say that $F$ is quadratically Lipschitz smooth in this case.
    
\end{definition}

\begin{definition} \label{def-grad-superquadratic-lip-cont}
    $\dot F$ is superquadratically Lipschitz continuous if there exists a sequence, $\lbrace \theta_k : k \in \mathbb{N} \rbrace$, for which 
    \begin{equation}
        \lim_k \inlinenorm{ \dot F(\theta_k)} = \infty \text{\quad and \quad } \limsup_k \inlinenorm{ \ddot F(\theta_k) }/ \inlinenorm{ \dot F(\theta_k) }^2 = \infty.
    \end{equation}
\end{definition}

\begin{definition} \label{def-grad-locally-lip-cont}
    $\dot F$ is locally Lipschitz continuous if for any compact set $C \subset \mathbb{R}^n$, $\exists \mathcal{L} \geq 0$ such that $\sup_{\psi \in C} \inlinenorm{ \ddot F(\psi)} \leq \mathcal{L}$. We say that $F$ is locally Lipschitz smooth. 
\end{definition}

These definitions can be related. In this context, any function $F$ that is globally Lipschitz smooth, globally Lipschitz Hessian smooth, subquadratically, quadratically, or superquadratically Lipschitz smooth is locally Lipschitz smooth \cite[see][Lemma B.1]{patel2023gradient}. Furthermore, functions that are subquadratically Lipschitz smooth \textit{cannot} be quadratically Lipschitz smooth; functions that are quadratically Lipschitz smooth cannot be subquadratically Lipschitz smooth; and, any function that is superquadratically Lipschitz smooth cannot be subquadratically or quadratically Lipschitz smooth (and vice versa). Therefore, we view \cref{def-grad-locally-lip-cont} as the least restrictive definition, and \cref{def-grad-globally-lip-cont} the most restrictive.

We now contextualize the above definitions of continuity with several examples. We will begin with an example from factor analysis, then analyze specific instances of feed forward neural networks, estimation in generalized estimating equations, and Inverse Gaussian regression. A short introduction will be provided for each, followed by an analysis of smoothness properties.

\subsection{Factor Analysis}\label{subsec:factor-analysis}
Factor analysis is a statistical method for data reduction and pattern extraction, often employed to identify a set of latent variables (usually, of lower dimension) to explain the observations. That is, given independent centered observations, $\lbrace X_1,\ldots,X_N \rbrace \subset \mathbb{R}^m$, a factor analysis posits $X_i \sim \mathcal{N}(0, [WW^\intercal + V]^{-1})$, where $W$ usually has fewer than $m$ columns and $V$ is a diagonal matrix. Then, a factor analysis proceeds by estimating $W$ and $V$ using maximum likelihood estimation.

Consider the case of just one centered observation, $x \in \mathbb{R}$; and, in our model, we set $W = 0$ and we let $\theta^2 := V$ be the only entry of $V$. Then, the negative log-likelihood for this model is the objective function: $F:\mathbb{R}_{> 0} \to \mathbb{R}$, given by
\begin{equation}
F(\theta) = \frac{1}{2} \log (2\pi) - \frac{1}{2} \log(\theta^2) + \frac{1}{2}\theta^2 x^2.
\end{equation}
A direct computation of the gradient and Hessian functions yields,
\begin{equation}
\dot F(\theta) = -\frac{1}{\theta} + \theta x^2,\quad\text{and}\quad \ddot F(\theta) = \frac{1}{\theta^2} + x^2.
\end{equation}

Using these calculations, it is easy to verify that $F$ is \textbf{not} globally Lipschitz smooth nor globally Lipschitz Hessian smooth. We show that $F$ is quadratically Lipschitz smooth. First, we determine all sequences $\lbrace \theta_k \rbrace$ such that $|\dot F(\theta_k) | \to \infty$. We see that this occurs for any sequence that tends to zero or infinity (when $x \neq 0$). Second, by direct calculation,
\begin{equation}
\frac{|\ddot F(\theta) |}{\dot F(\theta)^2} = \frac{\theta^{-2} + x^2}{\theta^{-2} - 2x^2 + \theta^2 x^2} = \frac{1 + \theta^2 x^2}{1 - 2x^2\theta^2 + \theta^4 x^2}.
\end{equation}
When $\lbrace \theta_k \rbrace$ tends to zero, this ratio converges to $1$. When $\lbrace \theta_k \rbrace$ tends to infinity, this ratio converges to $0$. Hence, $F$ is quadratically Lipschitz smooth.

\subsection{Feed Forward Neural Network}\label{subsec:classification-nn}
In this section, we focus on the simplified example of a one dimensional four layer feed forward neural network with linear activations except for the last layer which is a sigmoid. Related information can be found in \cite[Appendix A.2]{patel2022globalconvergence}. The dataset for classification is $\{(1,1), (0,0)\}$ (each occurring with equal probability), where the first element of each tuple is the label, and the second element is the feature. We train the network using binary cross-entropy loss. The resulting risk function, which is the objective function, is (up to an additive constant)
\begin{equation} \label{eq:binary-cross-entropy}
    F(\theta) = \log(1+\exp\bigpar{{-w_4w_3w_2w_1}}),
\end{equation}
where $\theta \in \mathbb{R}^4$ is the vector whose entries are $w_1,w_2,w_3,$ and $w_4$.

The gradient and hessian of \cref{eq:binary-cross-entropy} are
\begin{equation}
    \begin{aligned}
        \dot{F}(\theta)
        &= \frac{-1}{\exp(w_4w_3w_2w_1) + 1} 
        \begin{bmatrix}
            w_4w_3w_2\\
            w_4w_3w_1\\
            w_4w_2w_1\\
            w_3w_2w_1
        \end{bmatrix},
    \end{aligned}
\end{equation}
and
\begin{equation}
    \begin{aligned}
        \ddot{F}(\theta) 
        &= \frac{-1}{\exp(w_4w_3w_2w_1)+1} 
        \begin{bmatrix}
            0 & w_3w_4 & w_4w_2 & w_3w_2 \\
            w_3w_4 & 0 & w_1w_4 & w_1w_3\\
            w_4w_2 & w_1w_4 & 0 & w_2w_1 \\
            w_3w_2 & w_1w_3 & w_2w_1 & 0
        \end{bmatrix}\\
        & + \frac{\exp(w_4w_3w_2w_1)}{\bigpar{\exp(w_4w_3w_2w_1) + 1}^2}
        \begin{bmatrix}
            w_4w_3w_2\\
            w_4w_3w_1\\
            w_4w_2w_1\\
            w_3w_2w_1
        \end{bmatrix} 
        \begin{bmatrix}
            w_4w_3w_2\\
            w_4w_3w_1\\
            w_4w_2w_1\\
            w_3w_2w_1
        \end{bmatrix}^\intercal
.
    \end{aligned}
\end{equation}

We now consider each type of smoothness outlined above. 
\begin{remunerate}
\item The Frobenius norm of the Hessian is at least the absolute value of its $(1,1)$ entry. This entry becomes $w_2^6/4$ when evaluated at $w_1 = 0$ and $w_2 = w_3 = w_4$ is $w_2^6/4$, which becomes unbounded as $w_2 \to \infty$. Hence, $F$ is \textbf{not} globally Lipschitz smooth.
\item $\ddot F(0) = 0$. Hence, using the Frobenius norm, $\inlinenorm{ \ddot F(\theta) - \ddot F(0) } = \inlinenorm{ \ddot F(\theta) }$, which is at least the absolute value of the $(1,1)$ entry of $\ddot F(\theta)$. Consider $\theta$ specified by $w_1 = 0$ and $w_2=w_3=w_4$, then
\begin{equation}
\frac{\inlinenorm{ \ddot F(\theta) - \ddot F(0) }}{ \sqrt{ 3w_2^2 - 0}} \geq \frac{w_2^6}{4 \sqrt{ 3w_2^2}},
\end{equation}
where we have used the Euclidean distance between $\theta$ and $0$ is the denominator. The right hand side of the inequality diverges as $w_2 \to \infty$. Thus, $F$ is \textbf{not} globally Lipschitz Hessian smooth.
\item We again use the Frobenius norm for the Hessian and the Euclidean distance for the vectors. Consider $\theta$ with $w_1 = 0$ and $w_2 = w_3 = w_4$. Then, just as before,
\begin{equation}
\frac{ \inlinenorm{ \ddot F(\theta) } }{ \inlinenorm { \dot F (\theta) }^2} \geq \frac{w_2^6}{w_2^6 } = 1.
\end{equation}
Hence, $F$ is \textbf{not} subquadratically Lipschitz smooth.
\item With the same norms as above, 
\begin{equation}
\inlinenorm{ \dot F(\theta) }^2 = \left( \frac{1}{1 + \exp(w_4 w_3 w_2 w_1)}  \right)^2 \left( (w_4 w_3 w_2)^2 + (w_4 w_3 w_1)^2 + (w_4 w_2 w_1)^2 + (w_3 w_2 w_1)^2 \right).
\end{equation}
In order for this quantity to diverge, $w_4w_3w_2w_1 \leq 0$. Moreover, at most only one of $\lbrace w_1,w_2,w_3,w_4 \rbrace$ can be zero, while at least one must diverge to infinity. 

Hence, since $w_4w_3w_2w_1 \leq 0$ for $\inlinenorm{ \dot F(\theta)}^2$ to diverge, then, eventually,
\begin{equation}
\inlinenorm{ \dot F(\theta) }^2 \geq \frac{(w_4 w_3 w_2)^2 + (w_4 w_3 w_1)^2 + (w_4 w_2 w_1)^2 + (w_3 w_2 w_1)^2}{2}.
\end{equation}

Moreover, using $w_4w_3w_2w_1 \leq 0$ and the triangle inequality,
\begin{equation}
\inlinenorm{ \ddot F(\theta) } \leq \sqrt{ \sum_{i \neq j} (w_i w_j)^2 } + (w_4 w_3 w_2)^2 + (w_4 w_3 w_1)^2 + (w_4 w_2 w_1)^2 + (w_3 w_2 w_1)^2.
\end{equation}

Hence, when $\inlinenorm{ \dot F(\theta) }$ diverges,
\begin{equation}
\frac{\inlinenorm{ \ddot F(\theta) } }{ \inlinenorm{ \dot F(\theta) }^2 }
\leq \frac{\sqrt{ \sum_{i \neq j} (w_i w_j)^2 }}{(w_4 w_3 w_2)^2 + (w_4 w_3 w_1)^2 + (w_4 w_2 w_1)^2 + (w_3 w_2 w_1)^2} + 2.
\end{equation}

Now, since at most three terms can diverge and at most one term can be zero, the limit of the right hand side is two. Hence, $F$ is quadratically Lipschitz continuous.
\end{remunerate}


\subsection{Generalized Estimating Equations}\label{subsec:gee}
We now consider objective functions arising in Generalized Estimating Equations. This modeling framework is a technique often employed in clustered data analysis, like repeated measurements in biomedical studies \cite[Chapter 3]{lipsitz2008geelongitudinal}. These methods often estimate parameters, $\theta$, by solving for a root in an \emph{estimating equation}; however, this estimation procedure is ambiguous when there are multiple roots, so it is convenient to consider the estimating equation as the gradient of a function (when applicable), and to instead treat estimation as a minimization problem. As a simple example, consider an observation $y \in \mathbb{R}_{> 0}$ estimated by a parameter $\theta > 0$ using the estimating equation,
\begin{equation} \label{eq-simple-ee}
\dot F(\theta) = -\frac{y - \sqrt{\theta}}{\theta}.
\end{equation}
Then, the root of this equation correspond to the minimizer of the function (up to an additive constant)
\begin{equation}
F(\theta) = -y \log( \theta) + 2 \sqrt{\theta},
\end{equation}
whose Hessian is
\begin{equation}
\ddot F(\theta) = \frac{ y - 0.5 \sqrt{\theta}}{\theta^2}.
\end{equation}

Using the above calculations, it is easy to verify that $F$ is \textbf{not} globally Lipschitz smooth, as the Hessian is unbounded as $\theta$ approaches $0$. Furthermore, for any $\theta \in (0, 4y^2)$,
\begin{equation}
\frac{| \ddot F(\theta) - \ddot F(4y^2) | }{|\theta - 4y^2|} = \frac{|\ddot F(\theta)| }{|\theta - 4y^2|} \geq \frac{ |y - 0.5 \sqrt{\theta}|}{4y^2\theta^2}.
\end{equation}
Therefore, $F$ is \textbf{not} globally Lipschitz Hessian smooth. We now show that $F$ is quadratically Lipschitz smooth. It is readily seen that
\begin{equation}
    \frac{|\ddot F(\theta)|}{\dot F(\theta)^2} = \frac{ |y - 0.5 \sqrt{\theta}|}{\theta^2} \frac{\theta^2}{(y - \sqrt{\theta})^2} = \frac{|y - 0.5 \sqrt{\theta}|}{(y - \sqrt\theta)^2}.
\end{equation}
The only sequence $\{\theta_k\}$ such that $|\dot F(\theta_k)|$ diverges are those that approach $0$. The preceding ratio tends to $1/y$ on those sequences. Therefore, $F$ is not subquadratically Lipschitz smooth (as $1/y > 0$), but is quadratically Lipschitz smooth.


\subsection{Inverse Gaussian Regression} \label{subsec:inv-gaussian-reg}
Lastly, in this section we consider an example from Inverse Gaussian regression. Inverse Gaussian modeling is a popular tool in many statistical applications, including survival analysis and reliability testing \cite{jayalath2020}. As a basic example, consider an observation $y > 0$ drawn from an Inverse Gaussian distribution with mean $\mu > 0$ and scale parameter $\phi = 1$. Consider the generalized linear model with canonical link and one parameter, $\theta \in \mathbb{R}_{<0}$ (i.e., only an intercept). The negative log-likelihood, and the first three derivatives are
\begin{equation}
    \begin{aligned}
        &F(\theta) = - ( \theta y + (-2\theta)^{1/2} ) - \frac{1}{2y} - \frac{1}{2} \ln(2\pi y^3),\\
        &\dot F(\theta) = - ( y - (-2\theta)^{-1/2} ), ~\ddot F(\theta) = (-2\theta)^{-3/2} 
        \text{, and } \dddot F(\theta) = 3(-2\theta)^{-5/2}.
    \end{aligned} 
\end{equation}
Using the above derivatives we now check the smoothness properties of $F$. It is apparent from the Hessian and third derivative that $F$ is \textbf{not} globally Lipschitz smooth \textbf{nor} globally Lipschitz Hessian smooth. We now show that $F$ is superquadratically Lipschitz smooth. It readily follows that
\begin{equation}
    \frac{|\ddot F(\theta)|}{\dot F(\theta)^2} = \frac{1}{y^2(-2\theta)^{3/2} - 2y(-2\theta) + (-2\theta)^{1/2}}.
\end{equation}
We must show that there exists a sequence $\{\theta_k\}$ such that $|\dot F(\theta_k)|$ and $|\ddot F(\theta_k)|$ diverge, and the preceding ratio also diverges. Any sequence such that $\theta_k \to 0$ satisfies this property. Hence, $F$ is \textbf{not} subquadratically or quadratically Lipschitz smooth, but is in fact superquadratically Lipschitz smooth.

We comment that the case of intercept only model with multiple data points with known shared scale parameter can be handled similarly. Furthermore, this example is notable in the sense that many statistical tests are done with reference to an intercept only model.



\subsection{In Summary}
In the examples above, it is shown that under simplified, yet realistic scenarios that common optimization problems arising in data science \emph{do not} always satisfy intermediate smoothness conditions, and commonly only satisfy the least restrictive definitions (\cref{def-grad-quadratic-lip-cont,def-grad-superquadratic-lip-cont,def-grad-locally-lip-cont}). We note that in more complex scenarios, one should cautiously assume that the continuity properties \emph{do not} become more favorable, and should be skeptical of the claimed convergence or complexity behavior of many optimization methods. In fact as we will shortly illustrate, by not considering the more appropriate and realistic smoothness conditions, the consequences can be catastrophic for many traditional and modern approaches to optimization in data science problems. 

\section{Divergence Examples for Gradient Methods without Objective Evaluations} \label{sec:catastophic-divergence}
In this section, we consider the consequences for optimization algorithm that \emph{never} evaluate the objective function when the gradient only satisfies the less restrictive local Lipschitz continuity property (see \cref{def-grad-locally-lip-cont}). We will show that such methods can experience \emph{catastrophic divergence}, where the optimality gap diverges, and the gradient will often stay bounded away from $0$. This illustrates that such methods can experience dire reliability problems under realistic conditions for data science problems. We will begin by constructing a class of objective functions in \cref{subsec:cd-preliminaries} which will often be employed in our constructions, then proceed to show that an extensive list of methods experience catastrophic divergence.

\subsection{Preliminaries} \label{subsec:cd-preliminaries}
As many optimization problems arising in data science can be formulated with objectives that are bounded below and have continuous derivatives (see \cref{sec:smoothness}), we will primarily construct objective functions to satisfy the following two conditions

\begin{property} \label{as-bounded below}
    The objective function, $F$, is bounded below by some constant $F_{l.b.} > -\infty$.
\end{property}

\begin{property} \label{as-loc-lip-cont}
    $\forall \theta \in \mathbb{R}^n$, the gradient function $\dot F(\theta) := \nabla F(\psi) \vert_{\psi=\theta}$ exists and is locally Lipschitz continuous.
\end{property}

While not always, we will often use the following function (as pieces of another function) to construct the objective in our examples. Let $m > 0$ and $d, \delta \in (0,1]$, and define

\begin{equation} \label{eqn-cd-building-function}
f(\theta; m, d, \delta) = \begin{cases}
-d\theta & \theta \in \left( 0, \frac{2-d}{16}m \right)  \\
\frac{8}{m}\left( \theta - \frac{m}{8} \right)^2 - m \left(\frac{-d^2 + 4d}{32}\right) & \theta \in \left[ \frac{2-d}{16}m, \frac{3}{16}m \right) \\
\frac{-5m}{16}\exp\left( \frac{5/16}{\theta/m - 1/2} + 1 \right) + m \left(\frac{11 + d^2 -4d}{32}  \right) & \theta \in \left( \frac{3}{16}m, \frac{1}{2}m \right) \\
m \left( \frac{11 + d^2 - 4d}{32}  \right) & \theta = \frac{1}{2}m \\
\frac{5m}{16} \exp\left(\frac{-5/16}{\theta/m - 1/2} + 1 \right) + m  \left(\frac{11 + d^2 -4d}{32}  \right) & \theta \in \left( \frac{1}{2}m, \frac{13}{16}m \right) \\
\frac{-8}{m} \left( \theta - \frac{7m}{8} \right)^2 + m\left( \frac{22 + d^2 -4d}{32}  \right) & \theta \in \left[ \frac{13}{16}m, \frac{\delta + 14}{16}m \right) \\
-\delta \theta + m \left(\frac{22 + d^2 + \delta^2 -4d + 28\delta}{32}\right) & \theta \in \left[ \frac{\delta + 14}{16}m, m  \right] .
\end{cases}
\end{equation}

This function was first constructed in \cite[eq. E.4]{patel2023gradient}; an example is plotted in \cref{plot-cd-building-function}. Importantly, this function was shown to satisfy \cref{as-bounded below,as-loc-lip-cont}, and have the following characteristics \cite[Proposition E.1]{patel2023gradient}.

\begin{figure}[!hb]
\centering
\begin{tikzpicture}[scale=.75]

    \draw[thick,->] (-.5,0) -- (8.5,0) node[anchor=north west] {$\theta$};
    \draw[thick,->] (0,-.5) -- (0,5.5) node[anchor=south east] {$f(\theta;m,1,1)$};


    \draw[thick] plot [smooth] coordinates{
        (0.0,-0.0)
(0.04020100502512563,-0.04020100502512563)
(0.20100502512562815,-0.20100502512562815)
(0.36180904522613067,-0.36180904522613067)
(0.5226130653266332,-0.5221017146031666)
(0.6834170854271356,-0.6497752582005505)
(0.8442211055276382,-0.7257329360369688)
(1.0050251256281406,-0.7499747481124214)
(1.1658291457286432,-0.7225006944269085)
(1.3266331658291457,-0.6433107749804298)
(1.4874371859296482,-0.5124049897729855)
(1.728643216080402,-0.2605926537073424)
(1.9698492462311559,0.016515615364500746)
(2.21105527638191,0.3199579518514699)
(2.3316582914572863,0.4813880057862896)
(2.4522613065326633,0.6487487816541662)
(2.5728643216080402,0.8211608671836848)
(2.7336683417085426,1.0562704215034684)
(2.8944723618090453,1.2918312770934373)
(3.0552763819095476,1.5180841975337684)
(3.2160804020100504,1.7199569204826775)
(3.3768844221105527,1.8770289162017866)
(3.5376884422110555,1.9695393003946937)
(3.698492462311558,1.9982970205337833)
(3.85929648241206,1.9999998694574241)
(4.0201005025125625,2.0)
(4.180904522613066,2.0000067690235555)
(4.341708542713568,2.0045170970424437)
(4.50251256281407,2.046948253736711)
(4.663316582914573,2.156821249250992)
(4.824120603015076,2.327179655585783)
(4.984924623115578,2.5368832743787517)
(5.1457286432160805,2.7666485380874297)
(5.306532663316583,3.0028331993572728)
(5.467336683417085,3.2367953122097046)
(5.628140703517588,3.463450355356745)
(5.788944723618091,3.6800420481485308)
(5.949748743718593,3.885276014608828)
(6.110552763819095,4.0787503551032165)
(6.2713567839195985,4.260592653707343)
(6.432160804020101,4.431232109727443)
(6.592964824120603,4.584322365596829)
(6.7537688442211055,4.6893702179237895)
(6.914572864321608,4.7427022044897855)
(7.075376884422111,4.744318325294816)
(7.236180904522613,4.69421858033888)
(7.396984924623116,4.592402969621979)
(7.557788944723618,4.442211055276382)
(7.71859296482412,4.28140703517588)
(7.8793969849246235,4.1206030150753765)
(8.0,4.0)
    };

\end{tikzpicture}
\caption{A plot of \cref{eqn-cd-building-function}.} \label{plot-cd-building-function}
\end{figure}

\begin{proposition} \label{result-cd-building-properties}
Let $m > 0$ and $d,\delta \in (0,1)$. The continuous extension of \cref{eqn-cd-building-function} to $[0,m]$ is continuous on its domain; bounded from below by $-m/8$; differentiable on $(0,m)$; admits one-sided derivatives of $-d$ at $\theta=0$ and $-\delta$ at $\theta = m$; has a locally Lipschitz continuous derivative function; and $f(m;m,d,\delta) \geq 7m/16$.
\end{proposition}

We will use \cref{eqn-cd-building-function} as a building block to define the following function used in many of our examples. Let $\lbrace S_{j} : j + 1 \in \mathbb{N} \rbrace$ be a strictly increasing sequence in $\mathbb{R}$ and let $\lbrace d_j \in (0,1]: j+1 \in \mathbb{N} \rbrace$. Let $F: \mathbb{R} \to \mathbb{R}$ be defined by
\begin{equation} \label{eqn-cd-objective}
F(\theta) = \begin{cases}
- d_0(\theta - S_0) & \theta \leq S_0 \\
f(\theta - S_0; S_1 - S_0, d_0, d_1) & \theta \in (S_0,S_1] \\
f(\theta - S_j; S_{j+1} - S_j, d_j, d_{j+1}) + F(S_j) & \theta \in (S_j, S_{j+1}], ~\forall j \in \mathbb{N}.
\end{cases}
\end{equation}
Note, \cref{eqn-cd-objective} is defined recursively as $F(S_j)$ is needed in order for $F$ to be defined on $(S_j, S_{j+1}]$. We now show that this objective satisfies \cref{as-bounded below,as-loc-lip-cont}, along with some other properties.

\begin{proposition} \label{result-cd-objective}
Let $\lbrace S_{j} : j + 1 \in \mathbb{N} \rbrace$ be a strictly increasing sequence in $\mathbb{R}$ and let $\lbrace d_j \in (0,1]: j + 1 \in \mathbb{N} \rbrace$. Let $F: \mathbb{R} \to \mathbb{R}$ be defined as in \cref{eqn-cd-objective}. Then, $F$ satisfies \cref{as-loc-lip-cont}. Moreover, $F(\theta) \geq 7(S_j - S_0)/16 - (S_{j+1} - S_j)/8$ for $\theta \in (S_j, S_{j+1}]$, $F(S_j) \geq 7 (S_j - S_0)/ 16$ and $\dot F(S_j) = -d_j$ for all $j +1 \in \mathbb{N}$.
\end{proposition}
\begin{proof}
By the properties of $-d_0\theta$ and $f(\theta; m, d, \delta)$ (see \cref{result-cd-building-properties}), we need only check the behavior of $F$ at the points $\lbrace S_j : j+1 \in \mathbb{N} \rbrace$. At $S_0$, $F(S_0) = 0$ and $\lim_{\theta \downarrow S_0} F(\theta) = \lim_{\theta \downarrow 0} f(\theta - S_0; S_1 - S_0, 1, d_1) = 0$. Moreover, the derivative from the left at $\theta = S_0$ is $-d_0$ (i.e, the derivative of $-d_0\theta$), and the derivative from the right is $-d_0$ by \cref{result-cd-building-properties}. In other words, $\dot F(S_0) = -d_0$. Moreover, the derivative is constant in a neighborhood of $\theta = S_0$, which implies that the derivative of $F$ is locally Lipschitz continuous at $\theta = S_0$. Finally, by \cref{result-cd-building-properties}, $F(\theta) \geq -(S_1 - S_0)/8$ on $\theta \in (S_0, S_1]$ and $F(S_1) \geq 7 (S_1 - S_0)/ 16$. 

Suppose for some $j \in \mathbb{N}$, we have verified, for $\ell=0,\ldots,j-1$, $F$ is continuously differentiable at $\theta = S_\ell$; the derivative is locally Lipschitz continuous at this point; $\dot F(S_\ell) = -d_\ell$; $F(\theta) \geq 7(S_{\ell+1} - S_0)/16 - (S_{\ell+1} - S_{\ell})/8$ for all $\theta \in (S_\ell,S_{\ell+1}]$; and $F(S_{\ell+1}) \geq 7 (S_{\ell+1} - S_0)/ 16$. 

Now, at $S_j$, $\lim_{\theta \downarrow S_j} F(\theta) = \lim_{\theta\downarrow S_j} f(\theta - S_j; S_{j+1} - S_j, d_{j}, d_{j+1}) + F(S_j) = F(S_j)$. So $F$ is continuous at $\theta = S_j$. By \cref{result-cd-building-properties}, the derivatives from the left and right at $\theta = S_j$ are both $-d_j$. As the derivative is constant in a small neighborhood of $\theta = S_j$, $\dot F$ is locally Lipschitz continuous at $\theta = S_j$. Finally, by \cref{result-cd-building-properties}, $F(\theta) \geq F(S_j) - (S_{j+1} - S_j)/8 \geq 7(S_j - S_0)/16 - (S_{j+1} - S_j)/8$ on $(S_j, S_{j+1}]$ and $F(S_{j+1}) \geq F(S_j) + 7(S_{j+1} - S_j)/16 \geq 7(S_{j+1} - S_0)/16$. This completes the proof by induction.
\end{proof}

With the necessary background for our constructed function, we now provide examples for all the algorithm listed in \cref{table-objective-free} in order, starting with gradient descent with constant step size.

\subsection{Constant Step Size} \label{subsec:ce-constant}
Consider gradient descent with constant step size: given $\theta_0 \in \mathbb{R}^n$, $\lbrace \theta_k : k \in \mathbb{N} \rbrace$ are recursively generated by
\begin{equation}
\theta_{k} = \theta_{k-1} - m \dot F(\theta_{k-1}), ~\forall k \in \mathbb{N}, 
\end{equation}
where $F:\mathbb{R}^n \to \mathbb{R}$. Gradient descent with constant step size is known to generate iterates on a convex, one-dimensional quadratic function that diverge (which produces an objective function sequence that diverges), if the step size exceeds a multiple of the reciprocal of the Lipschitz rank of the gradient function. Here, we show a slightly stronger statement.

\begin{proposition}
Let $F(\theta) = \theta^4/4$, which satisfies \cref{as-bounded below,as-loc-lip-cont}. Then, for gradient descent with any constant step size $m > 0$, $\exists \theta_0 \in \mathbb{R}$ such that the iterates of gradient descent diverge, the objective function at the iterates diverges, and the absolute value of the gradient function at the iterates diverges.
\end{proposition}
\begin{proof}
Let $m > 0$. Choose $\theta_0^2 > 2/m$. We begin by induction to show $\theta_{k}^2 > \theta_{k-1}^2$ for all $k \in \mathbb{N}$. As the base case and generalization follow the same strategy, we state the induction hypothesis and generalization step. The induction hypothesis is $\theta_{k}^2 > \theta_{k-1}^2 > 2/m$, for some $k \in \mathbb{N}$. We now show $\theta_{k+1}^2 > \theta_k^2$. Note, $\theta_{k+1} = \theta_k - m \theta_k^3$. If $\theta_{k} > 0$, then $\theta_{k+1} < \theta_{k}$, and
\begin{equation}
\theta_k^2 > 2/m \Leftrightarrow m\theta_k^3 > 2 \theta_k \Leftrightarrow -(\theta_k - m\theta_k^3) > \theta_k.
\end{equation}
Hence, if $\theta_k > 0$, then $\theta_{k+1}^2 > \theta_{k}^2 > 2/m$. If $\theta_k < 0$, then $\theta_{k} < \theta_{k+1}$, and
\begin{equation}
\theta_{k}^2 > 2/m \Leftrightarrow m \theta_k^3 < 2 \theta_k \Leftrightarrow - \theta_k < \theta_k - m \theta_k^3.
\end{equation}
Hence, if $\theta_k < 0$, then $\theta_{k+1}^2 > \theta_k^2 > 2/m$. 

With this established, we see that $\lbrace | \theta_k | \rbrace$ is a strictly increasing sequence. Therefore, $\lbrace F(\theta_k) = \theta_k^4 / 4 \rbrace$ is a strictly increasing sequence. Moreover,  $\lbrace |\dot F(\theta_k) | = |\theta_k|^3 \rbrace$ is a strictly increasing sequence.
\end{proof}

\subsection{Barzilai-Borwein Method} \label{subsec:ce-bbb}
We construct a one-dimensional objective function on which the iterates generated by the Barzilai-Borwein method will diverge and cause the optimality gap to diverge.\footnote{The authors did not know of a different example in \cite{malitsky2023adaptive} when creating the manuscript.} In this context, the Barzilai-Borwein method begins with an iterate $\theta_0$ and an initial step $m_0 > 0$ and generates iterates $\lbrace \theta_{k} : k \in \mathbb{N} \rbrace$ by the recursion $\theta_{k+1} = \theta_k - m_k \dot F(\theta_k)$ for $k + 1 \in \mathbb{N}$, where $F: \mathbb{R} \to \mathbb{R}$; and, for $k \in \mathbb{N}$, 
\begin{equation}
m_k = \frac{\theta_{k} - \theta_{k-1}}{\dot F(\theta_k) - \dot F(\theta_{k-1})}.
\end{equation}
Note, both variations of the Barzilai-Borwein method reduce to the same case in one dimension. We now construct our objective function, and show the Barzilai-Borwein method will produce iterates whose optimality gap diverges. 

\begin{proposition}
Let $m_0 > 0$, $S_j = m_0 j$ for all $j+1 \in \mathbb{N}$, and let $d_j = 2^{-j}$ for all $j + 1 \in \mathbb{N}$. Let $F: \mathbb{R} \to \mathbb{R}$ be defined as in \cref{eqn-cd-objective}. Let $\theta_0 = 0$, $\theta_1 = \theta_0 - m_0 \dot F(\theta_0)$, and
\begin{equation}
\theta_{k+1} = \theta_k - \frac{\theta_k - \theta_{k-1}}{\dot F(\theta_k) - \dot F(\theta_{k-1})} \dot F(\theta_k), ~k \in \mathbb{N}.
\end{equation}
Then $\lim_k F(\theta_k) = \infty$.
\end{proposition}
\begin{proof}
The initial step is a special case. We then proceed by induction on the second update. For the initial step,
$\theta_1 = \theta_0 - m_0 \dot F(\theta_0) = 0 + m_0 = m_0 = S_1$. 
The second step is the base case for our proof by induction. By \cref{result-cd-objective},
\begin{equation}
\theta_2 = \theta_1 - \frac{\theta_1 - \theta_{0}}{\dot F(\theta_1) - \dot F(\theta_{0})} \dot F(\theta_1) = m_0 - \frac{m_0}{-2^{-1} + 1}(- 2^{-1}) = 2m_0 = S_2.
\end{equation}
Suppose $\theta_\ell = m_0 \ell = S_{\ell}$ for all $\ell = 0,\ldots,j$. Then, by \cref{result-cd-objective},
\begin{equation}
\theta_{j+1} = \theta_j - \frac{\theta_j - \theta_{j-1}}{\dot F(\theta_j) - \dot F(\theta_{j-1}) } \dot F(\theta_j) = m_0 j - \frac{m_0}{-2^{-j} + 2^{-j+1}} (-2^{j}) = m_0 j + \frac{m_0}{2-1} = m_0 (j+1) = S_{j+1}.
\end{equation}
By induction, $\theta_j = m_0 j = S_j$ for all $j + 1\in \mathbb{N}$. By \cref{result-cd-objective}, $\lim_{j} F(\theta_j) = \lim_{j} F(S_j) \geq \lim_{j} 7S_j / 16 = \infty$.
\end{proof}

\subsection{Nesterov Accelerated Gradient Descent}\label{subsec:ce-nag}
Nesterov's accelerated gradient descent \cite{nesterov1983}, is known to achieve the optimal rate of convergence among all methods only using gradient information \cite[Chapter 2]{nesterov2004convexopt}. In this section, we focus on a slight variation of the original method discussed in \cite{li2023convex}, as it was shown under relaxed Lipschitz smoothness conditions that the modified algorithm retains the optimal rate of convergence.

\begin{algorithm}[H]
    \caption{Nesterov's Accelerated Gradient Method as specified in \cite[\S 4.2]{li2023convex} } 
    \label{alg:NAG}
    \begin{algorithmic}[1]
        \Require $\theta_0 \in \mathbb{R}^n, m \in (0, \infty)$
        \State $z_0 = \theta_0, B_0 = 0, A_0 = 1/m$
        \For{$t = 0,...$} 
            \State $B_{t+1} = B_t + .5(1+\sqrt{4B_t+1})$
            \State $A_{t+1} = B_{t+1} + \frac{1}{m}$
            \State $y_t = \theta_t + (1-\frac{A_t}{A_{t+1}})(z_t-\theta_t)$
            \State $\theta_{t+1} = y_t - m\dot{F}(y_t)$
            \State $z_{t+1} = z_t - m(A_{t+1}-A_t)\dot{F}(y_t)$
        \EndFor
    \end{algorithmic}
\end{algorithm}

Before constructing the objective function and showing the divergence behavior of \cref{alg:NAG}, we note some properties of $\lbrace B_t \rbrace$. First, $B_0 = 0$, $B_1 = 1$, and $B_{t+1} - B_t > 1$ for all $t \in \mathbb{N}$. As a result, $A_{t+1} - A_t = B_{t+1} - B_t > 1$ and $A_t/A_{t+1} \in (0,1)$ for all $t + 1 \in \mathbb{N}$. We now define three sequences that mimic the behavior of \cref{alg:NAG}. Let $m > 0$, $\Theta_0 = 0$, $Z_0 = 0$, $S_0 = 0$, $\Theta_t = Y_{t-1} + m$ for $t \in \mathbb{N}$, $Z_t = Z_{t-1} + m (A_{t} - A_{t-1})$, and
\begin{equation}
Y_t = \Theta_t + \left( 1 - \frac{A_{t-1}}{A_t} \right) (Z_t - \Theta_t), ~t\in \mathbb{N}.
\end{equation}
For these sequences, we have the following property.
\begin{lemma} \label{result-cd-nag-sequence}
For $t > 1$, $\Theta_t < Y_t < Z_t$. $\lbrace Y_t : t+1 \in \mathbb{N} \rbrace$ is a strictly increasing sequence. 
\end{lemma}
\begin{proof}
We begin with the first statement.
Note, $\Theta_0 = Z_0 = Y_0 = 0$ and $\Theta_1 = Z_1 = Y_1 = m$. We proceed by induction with the base case of $t = 2$. $\Theta_2 = Y_1 + m = 2m$; $Z_2 = m + m (A_2 - A_1) > \Theta_2$. As $S_2$ is a strict convex combination of $\Theta_2$ and $Z_2$, $\Theta_2 < Y_2 < Z_2$. Suppose, for all $\ell \in \mathbb{N}$ strictly less than $t \in \mathbb{N}_{\geq 3}$, $\Theta_\ell < Y_\ell < Z_\ell$. Then,
\begin{equation}
Z_t = Z_{t-1} + m(A_t - A_{t-1}) > Y_{t-1} + m = \Theta_t.
\end{equation}
Since $Y_t$ is a strict convex combination of $\Theta_t$ and $Z_t$, the statement holds.

Now, we verify that $\lbrace Y_t : t+1 \in \mathbb{N} \rbrace$ is strictly increasing. Note, $Y_0 = 0$, $Y_1 = 1$ and $Y_2 > \Theta_2 = Y_1 + m > Y_1$. Suppose this holds up to and including some index $t \in \mathbb{N}$. Then, $Y_{t+1} > \Theta_{t+1} = Y_{t} + m > Y_t$. The conclusion follows.
\end{proof}

We now construct the counterexample and show that \cref{alg:NAG} produces iterates that diverge and whose optimality gap diverges, and whose derivates at $\lbrace y_t \rbrace$ and at $\lbrace \theta_t \rbrace$ remain bounded away from zero. To do so, let $S_0 = \Theta_0 = 0$, $S_1 = Y_1 = m$,
\begin{equation} \label{eqn-cd-nag-sequence}
S_{2j} = \Theta_{j+1}, \quad\text{and}\quad S_{2j+1} = Y_{j+1}, ~\forall j \in \mathbb{N}.
\end{equation} 
By \cref{result-cd-nag-sequence}, then $\lbrace S_j : j+1 \in \mathbb{N} \rbrace$ is a strictly increasing sequence. We now have the following example.

\begin{proposition}
Let $m > 0$. Define $\lbrace S_j : j + 1 \in \mathbb{N} \rbrace$ as in \cref{eqn-cd-nag-sequence}, and let $d_j = 1$ for all $j + 1 \in \mathbb{N}$. Let $F: \mathbb{R} \to \mathbb{R}$ be defined as in \cref{eqn-cd-objective}. Let $\theta_0 = 0$, and let $\lbrace \theta_j : j \in \mathbb{N} \rbrace$ and $\lbrace y_j : j+1 \in \mathbb{N} \rbrace$ be defined as in \cref{alg:NAG}. Then, $\lim_k F(\theta_k) = \infty$, $\lim_k F(y_k) = \infty$, and $\inf_{k} | \dot F(\theta_k) | = \inf_k |\dot F(y_k) | = 1$.
\end{proposition}
\begin{proof}
We need to verify, $\theta_{j} = S_{2j}$ for all $j+1 \in \mathbb{N}$ and $\theta_{j+1} = S_{2j+1}$ for all $j+1 \in \mathbb{N}$. Once this is verified, then, by \cref{result-cd-objective}, $F(S_j) \geq 7S_j/16$ and $\dot F(S_j) = -1$ for all $j+1 \in \mathbb{N}$. The result will follow.

To achieve this, we need to show $\Theta_j = \theta_j$, $Y_j = y_j$ and $Z_j = z_j$ for all $j+1 \in \mathbb{N}$. We proceed by induction. When $\theta_0 = 0$, then $z_0 = 0$ and $y_0 = 0$. Thus, $\Theta_0 = \theta_0$, $Y_0 = y_0$ and $Z_0 = z_0$. Suppose for $j \in \mathbb{N}$, $\Theta_j = \theta_j$, $Y_j = y_j$ and $Z_j = z_j$. Then, $\theta_{j+1} =  y_j - m \dot F(y_j) = Y_j - m \dot F(Y_j) = Y_j - m \dot F(S_{2j-1}) = Y_j + m = \Theta_{j+1}$, since $\dot F(S_{2j-1}) = -1$ by \cref{result-cd-objective} and our choice of $d_{2j-1} = 1$. Moreover, $z_{j+1} = z_j - m (A_{j+1} - A_j) \dot F(Y_j) = Z_j + m(A_{j+1} - A_j) = Z_{j+1}$. Finally, $y_{j+1} = \theta_{j+1} + (1 - A_{j+1}/A_{j+2})(z_{j+1} - \theta_{j+1}) = \Theta_{j+1} + (1 - A_{j+1}/A_{j+2}) (Z_{j+1} - \Theta_{j+1}) = Y_{j+1}$. We have concluded the proof by induction. The result follows.
\end{proof}

\subsection{Bregman Distance Method}\label{subsec:ce-bregman}
Roughly, the Bregman Distance Method optimizes a composite nonconvex objective function, where one component of the objective is differentiable (to be denoted by $F_1(\theta)$), and the other is any proper, lower semicontinuous function (to be denoted by $F_2(\theta)$) \cite[see][for a discussion]{bolte2018first}. The insight for this method is to adapt to the ``geometry'' of the problem, and introduce a suitable Bregman distance function to create upper bound models on the objective. The parameter that regulates the step size is then the weight placed on the Bregman distance penalty, which we denote by $m$.

For our example, let $m > 0$. Define $S_j \defeq j m$ for all $j+1 \in \mathbb{N}$, and $d_{j} = 1$ for all $j+1 \in \mathbb{N}$. Let $F_1(\theta) \defeq 0$, and let $F_2: \mathbb{R} \to \mathbb{R}$ be defined as in \cref{eqn-cd-objective}. Then, our composite objective function is $F(\theta) \defeq F_1(\theta) + F_2(\theta) = F_2(\theta)$. We can use $.5\inlinenorm{\cdot}_2^2$ to define the Bregman distance (as it is convex and $F_1(\theta) = 0$), which reduces the Bregman distance method to the standard proximal procedure as specified in \cref{alg:bregman-dist}. Note, that the first order Taylor approximation of $F_1(\theta)$ is $0$, therefore \cref{alg:bregman-dist} is indeed the Bregman distance method.

\begin{algorithm}
    \caption{Bregman Distance Method \cite[NoLips in][\S 3]{bolte2018first}}
    \label{alg:bregman-dist}
    \begin{algorithmic}
        \Require $\theta_0 \in \mathbb{R}$, $m > 0$
        \For{$k = 1,...$}
            \State $\theta_k = \arg\min_{\theta \in \mathbb{R}} \left\{ F(\theta) + \frac{1}{2m}\norm{ \theta - \theta_{k-1} }_2^2 \right\}$
        \EndFor
    \end{algorithmic}
\end{algorithm}

In \cref{alg:bregman-dist}, the inner loop part is as difficult as the original problem. Hence, we assume that we can solve it locally to our preference. In this case, we have the following result.

\begin{proposition}
Let $m > 0$. Let $S_j = m_j$ and $d_j = 1$ for all $j+1 \in \mathbb{N}$. Let $F: \mathbb{R} \to \mathbb{R}$ be defined as in \cref{eqn-cd-objective}. If $\theta_0 = 0$ and the inner loop can be solved to a local minimizer, then there is a choice of local minimizers such that $\lim_k F(\theta_k) = \infty$ and $\inf_k |\dot F(\theta_k)| = 1$.
\end{proposition}
\begin{proof}
We proceed by induction to show $S_j = \theta_j$. First, $S_0 = \theta_0 = 0$. Suppose the relationship holds up to some $j \in \mathbb{N}$; that is, $\theta_j = S_j$. Then, we locally minimize $F(\theta) + \frac{1}{2m}(\theta - \theta_j)$.
We now verify that $S_{j+1}$ is a local minimizer and let $\theta_{j+1} = S_{j+1}$, which will complete the proof by induction. We check second-order optimality conditions. $\dot F(S_{j+1}) = \dot F(S_{j+1}) + m^{-1}(S_{j+1} - S_j) = -1 + m^{-1}m = 0$. Since $\dot F(\theta)$ is locally linear at $S_{j+1}$, the second derivative of the local problem is $1/m > 0$. Hence, we can set $\theta_{j+1} = S_{j+1}$. The result follows by \cref{result-cd-objective}.
\end{proof}

\subsection{Negative Curvature Method}\label{subsec:ce-negative}
Negative curvature methods try to exploit negative curvature detected through the Hessian to escape undesirable stationary points, like local maxima and saddlepoints when minimizing an objective. An example procedure explored in \cite{curtis2018exploiting} is shown in \cref{alg:neg-curv-two-step}.
\begin{algorithm}[H]
    \caption{Specific Two-step method presented in \cite[\S 2.1]{curtis2018exploiting}}
    \label{alg:neg-curv-two-step}
    \begin{algorithmic}
        \Require $\theta_0, m > 0, m' > 0$
        \For{$k = 0,...$}
            \If{$\ddot{F}(\theta_k) \succcurlyeq 0$}
                \State $s_k' = 0$
            \Else
                \State $s_k' = $ \Call{SelectDirection()}{}
            \EndIf
            \If{$\dot{F}(\theta_k) = 0$}
                \State $s_k = 0$
            \Else
                \State $s_k = -\dot{F}(\theta_k)$
            \EndIf
            \If{$s_k = s_k' = 0$} 
                \State \Return{$\theta_k$}
            \EndIf
            \State $\theta_{k+1} = \theta_k + m s_k + m' s_k'$
        \EndFor
    \end{algorithmic}
\end{algorithm}

The method in \cref{alg:neg-curv-two-step} is a two-step algorithm that alternates between negative curvature and gradient steps with pre-defined constant step size $m,n \in (0, \infty)$. Note that in \cref{alg:neg-curv-two-step}, $s_k'$ is selected to be a descent direction and a direction of negative curvature  \cite[see][\S 2 for details]{curtis2018exploiting}. Furthermore, in full generality, $s_k$ needs only to be a descent direction; however, as a simplification, we select the negative gradient. We now provide an example of catastrophic divergence.

\begin{proposition}\label{prop:neg-div}
    Let $m, m' > 0$. Let $S_j = mj$ and $d_j = 1$ for all $j+1 \in \mathbb{N}$. Define $F:\mathbb{R} \to \mathbb{R}$ as in \cref{eqn-cd-objective}. Let $\theta_0 = 0$ and let $\lbrace \theta_j : j \in \mathbb{N} \rbrace$ be generated by \cref{alg:neg-curv-two-step}. Then,, $\lim_{j} F(\theta_j) = \infty$, and $\inf_j |\dot{F}(\theta_j)| = 1$
\end{proposition}
\begin{proof}
We proceed by induction to show that $\theta_j = S_j$ and the result follows by \cref{result-cd-objective}. First, $\theta_0 = 0 = S_0$. Suppose this relationship holds up to some $j \in \mathbb{N}$. Since $F$ is locally linear around $S_j$, $\ddot F(S_j) = 0$ and $s_j' = 0$. Since $\dot F(S_j) = -1$ then $s_j = 1$. Hence, $\theta_{j+1} = S_j + m + 0 = S_{j+1}$.
\end{proof}

\subsection{Lipschitz Approximation}\label{subsec:ce-lipschitz}
Lipschitz approximation methods either use a local model to estimate the Lipschitz rank of the objective, or use gradient information inbetween iterates to form an estimate. \Cref{alg:lip-approx} does the latter, and was introduced in \cite[\S 2]{malitsky2020adaptive} for convex optimization. We now provide an example in which the optimality gap diverges.

\begin{algorithm}
    \caption{Lipschitz Approximation Method in \cite[\S 2]{malitsky2020adaptive}}
    \label{alg:lip-approx}
    \begin{algorithmic}
        \Require $\theta_0 \in \mathbb{R}^n$, $m_0 > 0$
        \State $w_0 = +\infty$
        \State $\theta_1 = \theta_0 - m_0 \dot{F}(\theta_0)$
        \For{$k = 1,...$}
            \State $m_k = \min\left\{ \sqrt{1+w_{k-1}} m_{k-1}, \frac{\norm{\theta_{k} - \theta_{k-1}}_2}{2\norm{\dot{F}(\theta_k) - \dot{F}(\theta_{k-1})}_2}  \right\}$
            \State $\theta_{k+1} = \theta_{k} - m_k \dot{F}(\theta_k)$
            \State $w_k = \frac{m_k}{m_{k-1}}$
        \EndFor
    \end{algorithmic}
\end{algorithm}

\begin{proposition}
Let $m_0 > 0$. Let $S_0 = 0$ and $S_{j+1} = S_j + m_0 (\sqrt{5}/2)^j$ for all $j+1 \in \mathbb{N}$. Let $d_{j} = (\sqrt{5}/[\sqrt{5} + 1])^j$ for all $j+1 \in \mathbb{N}$. Let $F: \mathbb{R}\to\mathbb{R}$ be defined by \cref{eqn-cd-objective}. If $\theta_0 = 0$ and $\lbrace \theta_j : j \in \mathbb{N} \rbrace$ are generated by \cref{alg:lip-approx}, then $\lim_{j} F(\theta_j) = \infty$.
\end{proposition}
\begin{proof}
Note, $\theta_0 = 0 = S_0$. We proceed by induction. For the base case, $j=1$, $\theta_1 = 0 - m_0 \dot F(S_0) = m_0 = S_1$. Moreover, $m_0 = m_0$.
Suppose, for $j+1 \in \mathbb{N}$, $\theta_{\ell} = S_\ell$ for all $\ell \in \lbrace 0,\ldots,j \rbrace$; and $m_{\ell} = m_0([\sqrt{5}+1]/2)^\ell$ for all $\ell \in \lbrace 0,\ldots, j-1 \rbrace$. We now conclude. First, we calculate $w_{j-1}$. By the induction hypothesis, (with $j > 1$)
\begin{equation}
w_{j-1} = \frac{m_0([\sqrt{5}+1]/2)^{j-1} }{m_0([\sqrt{5}+1]/2)^{j-2}} =  \frac{\sqrt{5} + 1}{2} =  \sqrt{ 1 + \frac{\sqrt{5} + 1}{2}} = \sqrt{1 + w_{j-1}}.
\end{equation}
Hence,
\begin{align}
m_j &= \min\left\lbrace \frac{\sqrt{5} + 1}{2} m_{j-1}, \frac{|\dot F(\theta_{j-1})|}{2|\dot F(\theta_{j}) - \dot F(\theta_{j-1})|} m_{j-1} \right\rbrace  
= \min\left\lbrace \frac{\sqrt{5} + 1}{2} m_{j-1}, \frac{d_{j-1}}{2(d_{j-1} - d_j)} m_{j-1} \right\rbrace \\
	&= \frac{\sqrt{5} + 1}{2} m_{j-1} 
	= \left(\frac{\sqrt{5} + 1}{2}\right)^{j} m_0. 
\end{align}
Finally,
\begin{equation}
\theta_{j+1} = \theta_j - m_j \dot F(\theta_j) = S_j + m_0 \left(\frac{\sqrt{5} + 1}{2}\right)^{j}\left( \frac{\sqrt{5}}{\sqrt{5} + 1} \right)^j = S_j + m_0 \left(\frac{\sqrt{5}}{2} \right)^j = S_{j+1}.
\end{equation}
To conclude, $\theta_j = S_j$ for all $j+1 \in \mathbb{N}$. By \cref{result-cd-objective}, $\lim_{j} F(\theta_j) = \infty$.
\end{proof}

\subsection{Weighted Gradient-Norm Damping}\label{subsec:ce-weighted}
In this section, we show that weighted gradient-norm damping methods such as \cref{alg:wngrad} from \cite{wu2020wngrad} will have a divergent optimality gap. The method from \cite{grapiglia2022AdaptiveTrust} can be handled similarly. These methods are very common in stochastic optimization settings.
\begin{algorithm}[H]
    \caption{WNGrad from \cite[\S 2]{wu2020wngrad}}
    \label{alg:wngrad}
    \begin{algorithmic}
        \Require $\theta_0 \in \mathbb{R}^n, b_0 > 0$
        \For{$k \in \mathbb{N}$}
            \State $\theta_{k} = \theta_{k-1} - \frac{1}{b_{k-1}} \dot{F}(\theta_{k-1})$
            \State $b_k = b_{k-1} + \frac{||\dot{F}(\theta_{k})||_2^2}{b_{k-1}}$
        \EndFor
    \end{algorithmic}
\end{algorithm}

\begin{proposition}
    Let $b_0 > 0$. Define $S_0 = 0$, and for all $j \in \mathbb{N}$, $S_j = S_{j-1} + \frac{1}{B_{j-1}}$ where $B_{0} = b_0$, and $B_{j} = B_{j-1} + \frac{1}{B_{j-1}}$. Let $d_j = 1$ for all $j+1 \in\mathbb{N}$. Let $F:\mathbb{R} \to \mathbb{R}$ be defined by \cref{eqn-cd-objective}. Let $\theta_0$ and $\{\theta_j : j \in \mathbb{N}\}$ be defined by \cref{alg:wngrad}. Then, $\lim_j F(\theta_j) = \infty$ and $\inf_j |\dot{F}(\theta_j)| = 1$. 
\end{proposition}
\begin{proof}
    First, we show that $\theta_j = S_j$ and $b_j = B_{j}$ by induction, and then show that $S_j$ diverges. To this end, $\theta_0 = S_0$ and $b_0 = B_{0}$ by assumption, therefore assume that $\theta_j = S_j$ and $b_j = B_j$ for some $j \in \mathbb{N} \cup \{0\}$. Then, since $\dot{F}(\theta_j) = \dot{F}(S_j) = -1$, $\theta_{j+1} = \theta_j - \frac{1}{b_j}\dot{F}(\theta_j) = S_j - \frac{1}{B_j}\dot{F}(S_j) = S_j + \frac{1}{B_j} = S_{j+1}$, and $b_{j+1} = b_{j} + \frac{\inlinenorm{\dot{F}(\theta_{j+1})}_2^2 }{b_{j}} = B_{j} + \frac{1}{B_{j}} = B_{j+1}$. Lastly, to show $S_j$ diverges we show that $1 \leq B_j \leq jB_1$ for $j \in \mathbb{N}$ by induction. The base case is true since either $b_0 > 1$, implying that $B_1 = b_0 + \frac{1}{b_0} \geq 1$, or $0 < b_0 \leq 1$ and $B_1 \geq \frac{1}{b_0} \geq 1$. Suppose this is true for $j \in \mathbb{N}$, then by the inductive hypothesis $\frac{1}{B_j} \leq 1 \leq B_1$ so that $B_{j+1} = B_j + \frac{1}{B_j} \leq jB_1 + B_1 = (j+1)B_1$. Lastly, $B_{j+1} \geq B_j \geq 1$ by the inductive hypothesis. Therefore, by definition $S_j = \sum_{i=0}^{j-1} \frac{1}{B_i} \geq \frac{1}{B_0}+\frac{1}{B_1}\sum_{i=1}^{j-1} \frac{1}{i}$ which implies $S_j$ diverges.
\end{proof}

\subsection{Adaptively Scaled Trust Region}\label{subsec:ce-adaptiveTR}
Adaptively scaled trust region methods are a very recent line of research that \emph{never} evaluates the objective function, but instead uses gradient information in deciding a trust region \cite[see][]{gratton2022firstorderOFFO,gratton2023OFFOSecondOrderOpt}. Notably, it was shown that a deterministic Adagrad-like method is a special case of the general framework \cite[see][\S 3 for a discussion]{gratton2022firstorderOFFO}. \Cref{algo:adagrad} outlines the deterministic Adagrad-like method, where the trust region elements are omitted as they are not necessary in this case \cite[see][\S 2 and \S 3]{gratton2022firstorderOFFO}. We now provide an example where this important, special case of adaptively scaled trust region methods will have a divergent optimality gap. Note, a variation of this method described in \cite[\S 4]{gratton2022firstorderOFFO} can be treated similarly. 

\begin{algorithm}
    \caption{Adagrad-like Method described in \cite[\S 3]{gratton2022firstorderOFFO} with $\vartheta = 1$}
    \label{algo:adagrad}
    \begin{algorithmic}
        \Require $\zeta \in (0, 1], \mu \in (0,1), \theta_0 \in \mathbb{R}^n$
        \For{$k = 0,1,2,...$}
            \State $w_k = \left( \zeta + \sum_{j=0}^k \dot F(\theta_j)^2 \right)^{\mu}$ \Comment{All operations are done element-wise}
            \State $\theta_{k+1} = \theta_k - \frac{1}{w_k} \dot F(\theta_k)$ \Comment{All operations are done element-wise}
        \EndFor
    \end{algorithmic}
\end{algorithm}

\begin{proposition}
Let $\zeta \in (0,1]$ and $\mu \in (0,1)$. Let $S_0 = 0$ and $S_j = S_{j-1} + (\zeta +j)^\mu$ for all $j \in \mathbb{N}$. Let $d_j = 1$ for all $j+1 \in \mathbb{N}$. Let $F: \mathbb{R} \to \mathbb{R}$ be defined by \cref{eqn-cd-objective}. Let $\theta_0 = 0$ and $\lbrace \theta_j : j\in\mathbb{N} \rbrace$ be defined by \cref{algo:adagrad}. Then, $\lim_{j} F(\theta_j) = \infty$ and $\inf_{j} |\dot F(\theta_j)| = 1$.
\end{proposition}
\begin{proof}
First, we verify $\theta_j = S_j$ by induction. Then, we verify $\lbrace S_j \rbrace$ diverges. For the base case, $S_0 = \theta_0$. Suppose $S_\ell = \theta_{\ell}$ for all $\ell \in \lbrace 0,\ldots, j \rbrace$. We now generalize to $j+1$. First, since $\dot F(\theta_{\ell}) = \dot F(S_{\ell}) = -1$, $w_j = (\zeta + \sum_{\ell=0}^j 1)^{\mu} = (\zeta + j + 1)^{\mu}$. Then, $\theta_{j+1} = \theta_j + (\zeta+j+1)^{-\mu} = S_j + (\zeta+j+1)^{-\mu} = S_{j+1}$. Hence, $\theta_{j} = S_j$ for all $j+1 \in \mathbb{N}$. Now, $S_j = \sum_{\ell=1}^j (\zeta + j + 1)^{\mu}$. Since $\mu \leq 1$, $\lbrace S_j \rbrace$ diverges. By \cref{result-cd-objective}, the result follows.
\end{proof}

\subsection{Polyak's Method}\label{subsec:cd-polyaks}
Lastly, we construct a one-dimensional example where applying gradient descent with Polyak's step size rule results in the iterates diverging, objective value diverging, and gradients that stay bounded away from 0. In this context, iterates $\{\theta_k : k \in \mathbb{N}\}$ are generated through $\theta_{k+1} = \theta_k - m_k \dot{F}(\theta_k)$, where $F : \mathbb{R} \to \mathbb{R}$; and, for all $k+1 \in \mathbb{N}$,
\begin{equation} 
    m_k = \frac{F(\theta_k) - F_{l.b.}}{\dot{F}(\theta_k)^2}.
\end{equation}
Here, we consider $F_{l.b.}$ to be the largest lower bound, that is $F_{l.b.} = \min_{\theta\in\mathbb{R}}F(\theta)$.

\begin{proposition}
    Let $S_0 = 0$, $S_1 = 1$, $\mathcal{O}_0 = 0$, and for all $j\in\mathbb{N}_{\geq 2}$, 
    \begin{equation}
        S_j = S_{j-1}+8\mathcal{O}_{j-1}+\frac{248}{2048}, \text{\quad} \mathcal{O}_{j-1} = \frac{1346}{2048} (S_{j-1} - S_{j-2}) + \mathcal{O}_{j-2}.
    \end{equation}
    Additionally, let $d_j = -\frac{1}{8}$ for all $j+1 \in \mathbb{N}$, and define $F : \mathbb{R} \to \mathbb{R}$ as in \cref{eqn-cd-objective}. Let $\theta_0 = 1$ and let $\{\theta_j : j \in \mathbb{N}\}$ be defined as
    \begin{equation} \label{eq:polyak-iterates}
        \theta_j = \theta_{j-1} - \frac{F(\theta_{j-1}) - F_{l.b.}}{\dot{F}(\theta_{j-1})^2} \dot{F}(\theta_{j-1}),
    \end{equation}
    then $\lim_{j\to\infty} \theta_j = \infty$, $\lim_{j\to\infty} F(\theta_j) = \infty$, and for all $j+1 \in \mathbb{N}$, $\inf_j |\dot{F}(\theta_j)| = \frac{1}{8}$.
\end{proposition}
\begin{proof}
    First, we verify that $\mathcal{O}_{j-1} = F(S_{j-1})$ for all $j \in \mathbb{N}$. Clearly, $\mathcal{O}_0 = 0 = F(0) = F(S_0)$ by our assumptions and definition of $F$. Suppose this is true for some $j-1 \in \mathbb{N} \cup \{0\}$, we will show it is true for $j$. Specifically, by definition of $\mathcal{O}_j$, $F$, and the inductive hypothesis
    \begin{equation}
        \mathcal{O}_j = \frac{1346}{2048}(S_{j}-S_{j-1}) + \mathcal{O}_{j-1} = f(S_{j}-S_{j-1};S_j-S_{j-1},d_{j-1},d_{j}) + F(S_{j-1}) = F(S_j).
    \end{equation}
    Having established this equality, we now verify that $F_{l.b.} = -\frac{31}{2048}$. Indeed, for $\theta \leq 0$, $F(\theta) \geq 0$, and for $\theta \in (S_0,S_1]$ the minimum value is $-\frac{31}{2048}$ attained at $1/8$. For $\theta \in (S_j,S_{j+1}]$, by \cref{eqn-cd-objective}, \cref{result-cd-building-properties}, the relationship $F(S_j) = \mathcal{O}_j$, and definition of $S_{j+1}$
    \begin{equation}
        F(\theta) = f(\theta-S_j; S_{j+1} - S_j, d_j, d_{j+1}) + F(S_j) \geq -(S_{j+1}-S_j)/8 + \mathcal{O}_j = -\frac{31}{2048}. 
    \end{equation}
    Now we show that $\theta_j = S_{j+1}$ for all $j+1 \in \mathbb{N}$. This is true by assumption for $\theta_0$, so by induction, suppose that $\theta_{j-1} = S_j$, then using \cref{eq:polyak-iterates}, the minimum value $F_{l.b.}$, $\dot{F}(S_j) = \dot{F}(\theta_j) = d_j = \frac{-1}{8}$, and $F(S_j) = \mathcal{O}_j$
    \begin{equation}
        \theta_{j} = \theta_{j-1} - m_{j-1} \dot{F}(\theta_{j-1}) = \theta_{j-1} + 8F(\theta_{j-1})+\frac{248}{2048} = S_{j} + 8F(S_{j}) + \frac{248}{2048} = S_{j+1}.
    \end{equation}
    Note, this proves that for all $j+1 \in \mathbb{N}$, $\dot{F}(\theta_j) = \frac{-1}{8}$. To finish, we show that $S_j \to \infty$ as $j \to \infty$. To do this we show that $S_{j+1}-S_j \geq 1$. For the base case, $S_{1} - S_{0} = 1$. Now, suppose for $j \in \mathbb{N}$ that $S_{j}-S_{j-1} \geq 1$, then by \cref{eqn-cd-objective} and \cref{eqn-cd-building-function}, $S_{j+1}-S_{j} = 8\mathcal{O}_{j}+\frac{248}{2048} 
    = 8\bigpar{ \frac{1346}{2048} (S_{j}-S_{j-1}) + \mathcal{O}_{j-1}}+\frac{248}{2048} 
    \geq 8\mathcal{O}_{j-1} +\frac{248}{2048} 
    = (S_{j}-S_{j-1}) \geq 1$. This implies that the iterates diverge, and the function values diverge.
\end{proof}

\section{Objective Function Evaluation Explosion with Objective Evaluations} \label{sec:exponential-increase}
Having illustrated that optimization algorithms that never evaluate the objective function can experience catastrophic divergence under the more realistic condition of local Lipschitz continuity of the gradient, we now try to understand the consequences for methods that utilize the objective to enforce descent. Naturally, these methods avoid catastrophic divergence, however we will show that the methods in \cref{table-multiple-objective-evaluation} will require an exponentially increasing number of objective function evaluations per iteration up to an arbitrary iterate $J \in \mathbb{N}$. This illustrates that in the more realistic regime for data science, this set of methods can face complexity challenges.

As in \cref{sec:catastophic-divergence}, we start with some preliminaries where we provide details on our construction. We then provide examples for each of the algorithms listed in \cref{table-multiple-objective-evaluation} in order, using our construction.

\subsection{Preliminaries} \label{subsec:exploding-obj-eval-preliminaries}
To construct the objective functions, we will make use of an interpolating polynomial of degree $9$ as a building block. The use of an interpolating polynomial will give us the freedom of setting the value of the function, the derivative, and second derivative at key points. We first show the existence of such a polynomial, then provide our general construction.

\begin{lemma} \label{lemma-ee-interp-poly}
    Let $m > 0$. Let $f_{-}, f_{0}, f_+, f_-', f_{0}', f_{+}', f_-'', f_{0}'', f_{+}'' \in \mathbb{R}$. Then, there exists a polynomial function of degree $9$, $P:\mathbb{R} \to \mathbb{R}$ such that: (1) $f_- = P(-m)$, $f_{0} = P(0)$, $f_+ = P(m)$; (2) $f_-' = \dot P(-m)$, $f_{0}' = \dot P(0)$, $f_+' = \dot P(m)$; and (3)  $f_-'' = \ddot P(-m), f_{0}'' = \ddot P(0), f_{+}'' = \ddot P(m)$.
\end{lemma}
\begin{proof}
    Let $P(\theta) = \sum_{i=0}^9 c_i \theta^i$. Then, from $P(0) = c_0, \dot P(0) = c_1, \ddot P(0) = 2c_2$ we set $c_0 = f_0, c_1 = f_0', c_2 = f_0''/2$. For the rest of the coefficients, $\{c_i\}_{i=3}^9$, we look at the set of solutions to the following linear system
    \begin{equation}
        \begin{bmatrix}
            f_+ - f_0 - f_0' m - (f_0''/2) m^2 \\
            f_{-} - f_0 + f_0' m - (f_0''/2) m^2 \\
            f_+' - f_0' - f_0''m \\
            f_{-}' - f_0' + f_0''m \\
            f_{+}'' - f_0''\\
            f_{-}'' - f_0''
        \end{bmatrix}
        =
        \begin{bmatrix}
            m^9 & m^8 & m^7 & m^6 & m^5 & m^4 & m^3 \\
            -m^9 & m^8 & -m^7 & m^6 & -m^5 & m^4 & -m^3 \\
            9m^8 & 8m^7 & 7m^6 & 6m^5 & 5m^4 & 4m^3 & 3m^2 \\
            9m^8 & -8m^7 & 7m^6 & -6m^5 & 5m^4 & -4m^3 & 3m^2 \\
            72m^7 & 56m^6 & 42m^5 & 30m^4 & 20m^3 & 12m^2 & 6m \\
            -72m^7 & 56m^6 & -42m^5 & 30m^4 & -20m^3 & 12m^2 & -6m
        \end{bmatrix}
        \begin{bmatrix}
            c_9\\
            c_8\\
            c_7\\
            c_6\\
            c_5\\
            c_4\\
            c_3
        \end{bmatrix}.
    \end{equation}
    The above coefficient matrix has row echelon form
    \begin{equation}
        \begin{bmatrix}
            m^9 & m^8 & m^7 & m^6 & m^5 & m^4 & m^3 \\
            0 & 2m^8 & 0 & 2m^6 & 0 & 2m^4 & 0 \\
            0 & 0 & -2m^6 & 2m^5 & -4m^4 & -4m^3 & -6m^2 \\
            0 & 0 & 0 & 4m^5 & 0 & 8m^3 & 0 \\
            0 & 0 & 0 & 0 & 8m^3 & 8m^2 & 24m \\
            0 & 0 & 0 & 0 & 0 & 16m^2 & 0
        \end{bmatrix}.
    \end{equation}
    Therefore, there is an infinite number of solutions to the linear system, any of which satisfy properties $(1), (2), \text{and } (3)$.
\end{proof}

Now, for any $J \in \mathbb{N}$, let $\{S_j : j \in \{0,1,...,J\} \}$ be a strictly increasing sequence and $\Delta > 0$ such that the intervals $\{ [S_j - \Delta, S_j + \Delta] : j \in \{0,1,...,J\} \}$ are disjoint. Let $\{(f_j, f_j', f_j'') \in \mathbb{R}^3 : j \in \{0,1,...,J\}\}$ be values we select later. For $j \in \lbrace 0,\ldots, J \rbrace$, let $P_j(\theta)$ be a polynomial of degree $9$ such that
\begin{equation}
0 = P_j(S_j - \Delta) = \dot P_j (S_j - \Delta) = \ddot P_j (S_j - \Delta) = P_j(S_j + \Delta) = \dot P_j (S_j + \Delta) = \ddot P_j(S_j + \Delta),
\end{equation}
and $P_j (S_j) = f_j$, $\dot P_j(S_j) = f_j'$ and $\ddot P_j(S_j) = f_j''$. By \cref{lemma-ee-interp-poly}, this polynomial exists. Using these polynomials, we define our objective function, $F_J : \mathbb{R} \to \mathbb{R}$ by
\begin{equation} \label{eqn-ee-objective}
    F_J(\theta) = 
    \begin{cases}
    P_j( \theta ) & \theta \in \left[S_j - \Delta, S_j + \Delta \right], ~j=0,\ldots,J \\
    0 & \text{otherwise}.
    \end{cases}
\end{equation}

This function has properties stated in \cref{prop-ee-properties}.
\begin{proposition} \label{prop-ee-properties}
    For any $J \in \mathbb{N}$, given a strictly increasing sequence $\{S_j : j \in \{0,1,...,J\} \}$, $\Delta > 0$ such that the intervals $\{ [S_j - \Delta, S_j + \Delta] : j \in \{0,1,...,J\} \}$ are disjoint. Let $\{(f_j, f_j', f_j'') \in \mathbb{R}^3 : j \in \{0,1,...,J\}\}$. Let $F_J: \mathbb{R} \to \mathbb{R}$ be defined as in \cref{eqn-ee-objective}. Then, 
    \begin{remunerate}
        \item For all $j \in \{0,1,...,J\}$, $F_J(S_j) = f_j, \dot F_J(S_j) = f_j', \ddot F_J(S_j) = f_j''$;
        \item For all $j \in \{0,1,...,J\}$, $F_J(S_j - \Delta) = \dot F_J(S_j - \Delta) = \ddot F_J(S_j - \Delta) = F_j(S_j + \Delta) = \dot F_j(S_j + \Delta) = \ddot F_j(S_j + \Delta) = 0$;
        \item $F_J$ satisfies \cref{as-bounded below,as-loc-lip-cont}.
    \end{remunerate}
    \end{proposition}
    \begin{proof}
    First, properties 1 and 2 follow directly by construction of $F_J(\theta)$ and $P_j(\theta)$, and \cref{lemma-ee-interp-poly}. To check the other properties, we begin by checking the boundary points in $F_J$. By construction, $P_j$ are zero at these points and have one sided derivatives of value zero. As $F_J$ is zero everywhere else, then the objective is continuous and differentiable everywhere. We now verify that the objective is locally Lipschitz continuous. $P_j$ are polynomial functions which are twice continuously differentiable, which implies that they are locally Lipschitz continuous. The zero function is twice continuously differentiable. Hence, we need only check the local Lipschitz continuity of the derivative at the boundary points. 
    
    By assumption, the intervals $\{ [S_j - \Delta, S_j + \Delta] : j \in \{0,1,...,J\} \}$ are disjoint and $\Delta > 0$. Therefore, there exists an $\epsilon_1 > 0$ such that for any boundary point $S_{j} - \Delta$, $F_J(\theta) = 0$ for all $\theta \in (S_{j} - \Delta - \epsilon_1, S_{j} - \Delta]$. Now, let $\epsilon_2 \in (0, 2\Delta]$. Let $L$ be the local Lipschitz rank of $\dot P_j$ within $[S_j - \Delta, S_j - \Delta +\epsilon_2)$. Then, for any $\theta \in (S_j - \Delta - \epsilon_1, S_j - \Delta]$ and $\psi \in [S_j - \Delta, S_j - \Delta + \epsilon_2)$,
    \begin{equation}
    |\dot F_J(\theta) - \dot F_J(\psi)| = |\dot P_j(\psi)| \leq L \left|\psi - S_j + \Delta \right| \leq L |\psi - \theta|.
    \end{equation}
    The boundary point of $S_j + \Delta$ can be treated the same. Finally, $F_J$ is bounded from below since $F_J$ is continuous everywhere and only nonzero on a compact set.
\end{proof}

An example illustrating the first four components of \cref{eqn-ee-objective} is plotted in \cref{ee-example-function}.
\begin{figure}[H]
    \centering
    \input{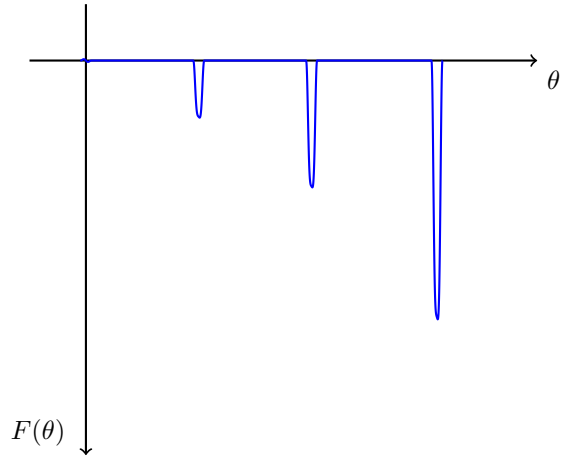}
    \caption{A plot of the first four components of \cref{eqn-ee-objective} for a specific choice of parameters. \label{ee-example-function}}
\end{figure}

We now show for any parametrization of the algorithms in \cref{table-multiple-objective-evaluation}, there exists an objective function through our construction such that the number of objective function evaluations is exponentially increasing between iterates. We go through the algorithms in \cref{table-multiple-objective-evaluation} in order, starting with Armijo Backtracking.

\subsection{Armijo's Backtracking Method} \label{subsec:armijo-backtracking}
Armijo backtracking method is a well-studied technique, that utilizes smoothness properties of the gradient to show that eventually a small step size must satisfy a descent condition derived from an upper bound model of the objective \cite{armijo1966minimization}. We present the pseudo-code in \cref{alg:armijo}.
\begin{algorithm}[H]
    \caption{Armijo's Backtracking Method}
    \label{alg:armijo}
    \begin{algorithmic}
        \Require $\alpha >0$, $\delta \in (0,1)$, $\rho \in (0,1)$, $\theta_0 \in \mathbb{R}^n$
        \For{$k = 0,1,2,...$}
        		$j = 0$
            \While{ $F(\theta_k - \alpha \delta^j \dot F(\theta_k) ) > F(\theta_k) - \rho\alpha \delta^j \inlinenorm{ \dot F(\theta_k)}_2^2$}
            		\State $j = j + 1$
            \EndWhile
            \State $\theta_{k+1} = \theta_k - \alpha \delta^j \dot F(\theta_k)$.
        \EndFor
    \end{algorithmic}
\end{algorithm}

We now show that there is an exponential increase in the number of objective function evaluations between accepted iterates up to $J$. 
\begin{proposition} \label{result-ee-armijo}
Let $\delta \in (0,1)$, $\alpha > 0$, $\rho \in (0,1)$, and $J \in \mathbb{N}$. Let $S_0 = 0$; for $j=1,\ldots,J$, let $S_j = S_{j-1} + \delta^{j - 2^{j-1}}$; and, for $j=0,\ldots,J$,
\begin{equation}
    ( f_j, f_j', f_j'' ) = 
    \begin{cases}
        \left( 0, -1/\alpha, 0 \right) & j = 0 \\
        \left( \frac{-\rho}{\alpha}\sum_{k=0}^{j-1} \delta^{2k+4 - 2^{k+2}}, -\frac{\delta^{j+2 - 2^{j+1}}}{\alpha}, 0 \right) & j = 1 \ldots, J
    \end{cases}.
\end{equation}
Let $\Delta = \frac{1-\delta}{2}$, and $F_J: \mathbb{R} \to \mathbb{R}$ be defined as in \cref{eqn-ee-objective}. Let $\theta_0 = 0$ and let $\lbrace \theta_1,\ldots,\theta_J \rbrace$ be generated by \cref{alg:armijo}. Then, the total number of objective function evaluations taken by \cref{alg:armijo} upon acceptance of iterate $\theta_j$ is $2^{j}$ for $j=1,\ldots,J$.
\end{proposition}
\begin{proof}
First, note that $S_j$ is a strictly increasing finite sequence, and that for all $j \in \mathbb{N}$, $\delta^{j-2^{j-1}} > 1-\delta > 0$ so that our constructed function satisfies the assumptions in \cref{prop-ee-properties}. Now to prove the statement in the proposition we proceed by induction. By construction, $\theta_0 = S_0$ so for the base case. Since $F(\theta_0) = F(S_0) = 0$, 
\begin{equation}
F(\theta_0 - \alpha \dot F(\theta_0)) = F(\theta_0 + 1) = F(S_1) = - \frac{\rho}{\alpha} = F(\theta_0) - \alpha \rho \dot F(\theta_0)^2.
\end{equation}
Hence, the algorithm accepts $\theta_1 = S_1$. Moreover, the algorithm computed $F(\theta_0)$ and $F(S_1)$ for line search, totaling to $2 = 2^{1}$ objective evaluations. If $J = 1$ we are done. Otherwise, we continue.

Suppose for $j \in \lbrace 1,\ldots, J-1 \rbrace$, $\theta_j = S_j$ and the total number of objective evaluations taken by the algorithm upon acceptance of $\theta_j$ is $2^{j}$. We now generalize by showing that
\begin{equation} 
\theta_j - \alpha \delta^{2^j - 1} \dot F(\theta_j) = S_j - \alpha \delta^{2^j - 1} \dot F(S_j) = S_j + \delta^{2^j - 1}\delta^{j+2-2^{j+1}} = S_{j} + \delta^{j +1 - 2^{j}} = S_{j+1}
\end{equation}
is the accepted iterate (i.e., $\theta_{j+1} = S_{j+1}$). Hence, we must show that $F(\theta_{j} - \alpha \delta^{\ell} \dot F(\theta_j) ) > F(\theta_j) - \rho \alpha \delta^{\ell} \dot F(\theta_j)^2$ for all $\ell \in \lbrace 0,1,\ldots,2^{j}-2 \rbrace$, and then the opposite holds at $\ell = 2^{j} - 1$. To show this, we first verify $\theta_j - \alpha \delta^{\ell} \dot F(\theta_j) \in [ S_{j+1} + \frac{1-\delta}{2}, S_{j+2} - \frac{1-\delta}{2} ]$ for $\ell \in \lbrace 0,1,\ldots,2^{j}-2 \rbrace$, which implies $F(\theta_j - \alpha \delta^{\ell} \dot F(\theta_j)) = 0$. As the sequence is decreasing with increasing $\ell$, it is enough to verify that the two end points are in the interval. For the smallest term (i.e., $\ell = 2^j - 2$),
\begin{equation}
\theta_j - \alpha \delta^{2^j - 2} \dot F(\theta_j) = S_j + \delta^{2^j - 2 + j + 2 - 2^{j+1}} = S_j + \delta^{j - 2^{j}} > S_{j} + \delta^{j +1 - 2^{j}}  + \frac{1-\delta}{2} = S_{j+1} + \frac{1-\delta}{2},
\end{equation}
where the inequality follows from $\delta^{j - 2^j} > 1/2$ for all $j \in \mathbb{N} \cup \lbrace 0 \rbrace$, which provides the necessary inequality when rearranged. For the largest term (i.e., $\ell = 0$), since $\delta^{j+1-2^j}> 1/2 > \frac{1-\delta}{2}$ for all $j \in \mathbb{N} \cup \{0\}$,
\begin{equation}
\theta_j - \alpha \dot F(\theta_j) =S_j + \delta^{j+2-2^{j+1}} = S_{j+2} + \delta^{j+2-2^{j+1}} - \delta^{j+1 - 2^{j}} - \delta^{j+2 - 2^{j+1}} \le S_{j+2} - \frac{1-\delta}{2}.
\end{equation}
Therefore, since for all $j \in \lbrace 1,\ldots, J-1 \rbrace$, $F(\theta_j) = P(S_j) < 0$, $\dot F(\theta_j) = \dot P(\theta_j) < 0$, and $\delta, \alpha, \rho > 0$ we have for all $\ell \in \{0,1,...,2^j-2\}$, $F(\theta_j - \alpha \delta^{\ell} \dot F(\theta_j)) = 0 > F(\theta_j) - \rho\alpha\delta^\ell\dot{F}(\theta_j)^2$. Lastly, for $\ell = 2^j-1$ (i.e., $\theta_j - \alpha \delta^{\ell} \dot F(\theta_j) = S_{j+1}$), since $\delta < 1$ we obtain that $\delta^{2j+4-2^{j+2}} \geq \delta^{2^j-1} \delta^{2j+4-2^{j+2}}$ for all $j \in \mathbb{N} \cup \{0\}$, which implies by adding the extra terms and multiplying by $-\rho/\alpha$ that
\begin{equation}
    F(S_{j+1}) = -\frac{\rho}{\alpha}\sum_{k=0}^j \delta^{2k+4-2^{k+2}} \leq -\frac{\rho}{\alpha}\sum_{k=0}^{j-1} \delta^{2k+4-2^{k+2}} - \frac{\rho}{\alpha} \delta^{2^j-1} \delta^{2j+4-2^{j+2}} = F(\theta_j) - \rho\alpha\delta^{2^j-1}\dot{F}(\theta_j)^2.
\end{equation}
Therefore, $2^j$ objective function evaluations were required to accept $\theta_{j+1}$ (excluding $F(\theta_j)$ as that was computed by the previous iteration). By the inductive hypothesis, $2^j$ total objective evaluations are required for $\theta_j$, thus the total number of objective evaluations taken by the algorithm upon acceptance of $\theta_{j+1}$ is $2^j + 2^j = 2^{j+1}$.
\end{proof}

\subsection{Newton's Method with Cubic Regularization} \label{subsec:cubic-newton-method}
In this section, we consider an example for Cubic Regularized Newton's Method from \cite{nesterov2006cubic} (see \cref{alg:crn}). At each iteration of this algorithm, a quadratic model of the objective plus a cubic penalty term is minimized, with the intuition being that a full unpenalized Newton's step is not suitable far away from stationary points and might not guarantee descent. The regularization parameter is selected by line search.
\begin{algorithm}[H]
    \caption{Cubic Regularized Newton's Method}
    \label{alg:crn}
    \begin{algorithmic}
        \Require $L_0 > 0$, $\delta_1 > 1$, $\theta_0 \in \mathbb{R}^n$
        \For{$k = 0,1,...$}
            \State $\ell \leftarrow 0$
            \State $M^k_\ell \leftarrow L_0$ 
            \While{$\text{true}$}
                \State $U_k(\psi; M^k_\ell) \leftarrow \dot F(\theta_k)^\intercal (\psi - \theta_k) + .5 (\psi - \theta_k)^\intercal \ddot F(\theta_k) (\psi - \theta_k) + \frac{1}{6}M^k_\ell \inlinenorm{\psi - \theta_k}_2^3$
                \State $\psi^k_\ell \leftarrow \arg\min_{\psi \in \mathbb{R}^n} U_k(\psi; M^k_\ell)$
                \If{$F(\psi^k_\ell) \leq F(\theta_k) + U_k(\psi^k_\ell; M^k_\ell)$}
                    \State $\textbf{exit loop}$
                \EndIf
                \State $M^{\ell+1}_k \leftarrow \delta_1 M^k_\ell$
                \State $\ell \leftarrow \ell + 1$
            \EndWhile
            \State $\theta_{k+1} \leftarrow \psi^k_\ell$
        \EndFor
    \end{algorithmic}
\end{algorithm}

We now provide an objective function such that there is an exponential increase in objective function evaluations between accepted iterates up to the $J$th acceptance.

\begin{proposition}
    Let $L_0 > 0$, $\delta_1 \in (1,\infty)$, and $J \in \mathbb{N}$. Set $\delta = \frac{1}{\sqrt{\delta_1}}$, and $S_0 = 0$; for $j = 1,...,J$, let $S_j = S_{j-1} + \delta^{j-2^{j-1}}$; and, for $j = 0,\ldots, J$,
    \begin{equation}
        (f_j, f_j', f_j'') = 
        \begin{cases}
            \left(0, -L_0/2, 0 \right) & j = 0 \\
            \left( -\frac{L_0}{3} \sum_{k=0}^{j-1} (\delta^{k+2-2^{k+1}})^3, - \frac{L_0}{2}\left(\delta^{j+2-2^{j+1}}\right)^2, 0 \right) & $j = 1,\ldots,J$
        \end{cases}.
    \end{equation}
    Let $\Delta = \frac{1-\delta}{2}$, and $F_J : \mathbb{R}\to\mathbb{R}$ be defined as in \cref{eqn-ee-objective}. Let $\theta_0 = 0$ and let $\{\theta_1,...,\theta_J\}$ be generated by \cref{alg:crn}. Then, the total number of objective function evaluations taken by \cref{alg:crn} upon acceptance of iterate $\theta_j$ is $2^j$ for $j = 1,...,J$.
\end{proposition}
\begin{proof}
    We first provide a global minimizer of $U_j(\psi; M)$ for any penalty parameter, $M > 0$, in the special case that $\theta_j$ satisfies $\ddot F(\theta_j) = 0$ and $\dot F(\theta_j) \not= 0$. Specifically, $U_j$, $\dot U_j$, and $\ddot U_j$ are
    \begin{equation}
        \begin{aligned}
            &U_j(\psi; M) = \dot F(\theta_j)^\intercal (\psi - \theta_j) +  
            \begin{cases}
                \frac{M}{6}(\psi-\theta_j)^3 & (\psi-\theta_j) \geq 0 \\
                -\frac{M}{6}(\psi-\theta_j)^3 & (\psi-\theta_j) < 0
            \end{cases},\\
            \text{\quad}
            &\dot U_j(\psi; M) = \dot F(\theta_j) + 
            \begin{cases}
                \frac{M}{2} (\psi-\theta_j)^2 & (\psi-\theta_j) \geq 0 \\
                -\frac{M}{2} (\psi-\theta_j)^2 & (\psi-\theta_j) < 0
            \end{cases},\\
            &\ddot U_j(\psi; M) = 
            \begin{cases}
                M(\psi-\theta_j) & (\psi-\theta_j) \geq 0\\
                -M(\psi-\theta_j) & (\psi-\theta_j) < 0
            \end{cases}.
        \end{aligned}
    \end{equation}
    From these equations, one can verify by checking first and second order sufficient conditions that two possible minimizers are $\theta_j + \sqrt{-2\dot F(\theta_j)/M}$ and $\theta_j - \sqrt{2\dot F(\theta_j)/M}$, only one of which is an element of $\mathbb{R}$ depending on the sign of $\dot F(\theta_j)$. 
    
    Next, the function $F_J$ has all the properties in \cref{prop-ee-properties} (see the proof of \cref{result-ee-armijo}). We now proceed by induction to prove the statement in the current proposition. For the base case, by construction $\theta_0 = S_0$, and by \cref{alg:acr}, $M^0_0 = L_0$. By the above discussion, this implies the first trial point is $\psi^0_0 = \theta_0 + \sqrt{-2\dot F(S_0)/L_0} = S_0 + 1 = S_1$. Therefore, 
    \begin{equation}
        F(S_1) = \frac{-L_0}{3} = F(\theta_0) - \sqrt{2/L_0}(2/3)(-\dot F(S_0))^{3/2} = F(\theta_0) + U_0(\psi_0^0; M_0^0).
    \end{equation}
    This implies that $\theta_1 = \psi^0_0 = S_1$. Upon acceptance of $S_1$, the algorithm took two objective function evaluation ($F(\theta_0)$ and $F(S_1)$). If $J=1$ we are done, otherwise suppose that for $j \in \{1,...,J-1\}$, $\theta_j = S_j$ and $2^j$ total objective function evaluations have been taken by the algorithm upon acceptance of $\theta_j$. Using the identity $\delta = \sqrt{1/\delta_1}$, we generalize this to $j+1$ by showing that
    \begin{equation}
        \theta_j + \sqrt{-2\dot F(\theta_j)/(\delta_1^{2^j-1}L_0)} = S_j + \sqrt{1/\delta_1^{2^j-1}} \delta^{j+2-2^{j+1}} = S_j + \delta^{2^j-1}\delta^{j+2-2^{j+1}} = S_{j+1}
    \end{equation}
    is the accepted iterate (i.e., $\theta_{j+1} = S_{j+1}$). We first show that for all $\ell \in \{0,...,2^{j}-2\}$ that $\psi^j_\ell \in [S_{j+1} + (1-\delta)/2, S_{j+2}-(1-\delta)/2]$, which will imply that $\{\psi^j_\ell\}_{\ell = 0}^{2^j-2}$ are all rejected. To prove this, we only check the cases $\ell = 0$ and $\ell = 2^{j}-2$, as the optimal solution to the subproblem with penalty parameter $\delta^\ell_1 L_0$ is $\theta_j + \sqrt{-2\dot F(\theta_j)/(\delta^\ell_1 L_0)} = \theta_j + \delta^\ell \sqrt{-2\dot F(\theta_j)/L_0}$ is a decreasing quantity. Now taking $\ell = 2^j-2$ (i.e., we are in the smallest case), using the inductive hypothesis and $\delta < 1$
    \begin{equation}
        \theta_j + \delta^{2^j-2} \sqrt{-2\dot F(\theta_j)/L_0} = S_j + \delta^{2^j-2}\delta^{j+2-2^{j+1}} \geq S_{j+1} + \frac{1-\delta}{2}.
    \end{equation}
    In the case $\ell = 0$ (i.e., the largest case), using the inductive hypothesis and $\delta < 1$
    \begin{equation}
        \theta_j + \sqrt{-2\dot F(\theta_j)/L_0}  = S_j + \delta^{j+2-2^{j+1}} = S_{j+2} - \delta^{j+1-2^j} \leq S_{j+2} - \frac{1-\delta}{2}.
    \end{equation}
    This proves that for all $\ell \in \{0,...,2^{j}-2\}$ that $\psi^{j}_\ell \in [S_{j+1} + (1-\delta)/2, S_{j+2}-(1-\delta)/2]$, and thus, $F(\psi^j_\ell) = 0$ by construction. To show that this results in rejection of all trial iterates $\{\psi^j_\ell\}_{\ell = 0}^{2^j-2}$, note that
    \begin{equation}
        F(\theta_j)+U_j(\psi^j_\ell; \delta^\ell_1 L_0) = F(\theta_j) - \sqrt{2/(\delta^\ell_1 L_0)} (2/3) (-\dot F(\theta_j))^{3/2}.
    \end{equation}
    By inductive hypothesis, $F(\theta_j) = F(S_j) < 0$, which by using the above equation implies that for all $\ell \in \{0,...,2^{j}-2\}$, $F(\psi^j_\ell) = 0 > F(\theta_j)+U_j(\psi_\ell^j; \delta^\ell_1 L_0)$. Lastly, we show that $F(\psi_{2^j-1}^j) = F(S_{j+1})$ is accepted. Since $\delta < 1$, $(\delta^{j+2-2^{j+1}})^3 \geq \delta^{2^j-1}(\delta^{j+2-2^{j+1}})^3$ for all $j \in \mathbb{N}\cup\{0\}$, which by adding the extra terms in the sum for $F(S_{j+1})$ and multiplying by $-L_0/3$ implies
    \begin{equation}
        \begin{aligned}
        F(S_{j+1}) = -\frac{L_0}{3} \sum_{k=0}^{j} (\delta^{k+2-2^{k+1}})^3 &\leq -\frac{L_0}{3} \sum_{k=0}^{j-1} (\delta^{k+2-2^{k+1}})^3 - \delta^{2^j-1}\frac{L_0}{3} (\delta^{j+2-2^{j+1}})^3 \\
        &= F(S_j) - \delta^{2^j-1} \sqrt{\frac{2}{L_0}} \frac{2}{3} (-\dot F(\theta_j))^{3/2}.
        \end{aligned}
    \end{equation} 
    Therefore, $S_{j+1}$ is accepted. To finish the proof, everytime the algorithm computes a trial iterate, $\psi^{j}_\ell$, $F(\psi^{j}_\ell)$ is computed once to check satisfaction of the second order cubic regularized upper bound model. Therefore, to accept $\theta_{j+1}$ it takes $2^j$ objective function evaluations (excluding $F(\theta_j)$ as that was computed in the previous iteration). By the inductive hypothesis, the algorithm already took $2^{j}$ for $\theta_j$, so the total is $2^{j+1}$.
\end{proof}

\subsection{Adaptive Cubic Regularization} \label{subsec:adaptive-cubic-reg}
The Adaptive Cubic Regularization method from \cite{cartis2011cubic1} is shown in \Cref{alg:acr}. This method is similar to Cubic Regularized Newton's method \cite{nesterov2006cubic}, except that instead of an explicit line search to set the penalty parameter and propose iterates, this method utilizes a model suitability criterion similar to trust region algorithms.

We now provide an example in \cref{prop:acr-example} such that \cref{alg:acr} requires an exponential increase in the number of objective function evaluations between accepted iterates up until the $J$th acceptance.
\bigskip
\begin{breakablealgorithm}
    \caption{Adaptive Cubic Regularization method}
    \label{alg:acr}
    \begin{algorithmic}[1]
        \Require $\theta_0 \in \mathbb{R}^n$, $\delta_1, \delta_2 \in \mathbb{R}$ such that $\delta_2 \geq \delta_1 > 1$, $\eta_1, \eta_2 \in \mathbb{R}$ such that $1 > \eta_2 \geq \eta_1 > 0$, $\sigma_0 > 0$.
        \Require $B_0 \in \mathbb{R}^{n \times n}$ and $\mathrm{UpdateHessianApproximation()}$

        \For{$k = 0,1,...$}
            \State $\ell \leftarrow 0$
            \State $\sigma_\ell^k \leftarrow \sigma_k$
            \While{$\text{true}$}
                \State $U_k(\gamma; \sigma_\ell^k) \leftarrow F(\theta_k) + \gamma^\intercal \dot F(\theta_k) + .5 \gamma^\intercal B_k \gamma^\intercal + \frac{1}{3}\sigma_\ell^k \inlinenorm{\gamma}_2^3$
                \State $m'_k \leftarrow \arg\min_{m \in \mathbb{R}_{+}} U_k(-m \dot F(\theta_k))$
                \State $\gamma'_k \leftarrow -m_k' \dot F(\theta_k)$
                \State \text{Compute $\gamma_\ell^k$ such that $U_k(\gamma_\ell^k; \sigma_\ell^k) \le U_k(\gamma'_k; \sigma_\ell^k)$}
                \State $\rho_\ell^k \leftarrow \frac{F(\theta_k) - F(\theta_k + \gamma_\ell^k)}{F(\theta_k) - U_k(\gamma_\ell^k; \sigma_\ell^k)}$
                \If{$\rho_\ell^k \geq \eta_1$}
                    \State $\gamma_k \leftarrow \gamma_\ell^k$
                    \State $\textbf{exit loop}$
                \EndIf
                \State $\sigma_{\ell+1}^{k} \in [\delta_1\sigma_\ell^k, \delta_2 \sigma_\ell^k]$
                \State $\ell \leftarrow \ell + 1$
            \EndWhile
        \State $\theta_{k+1} = \theta_{k} + \gamma_k$
        \If{$\rho_k > \eta_2$}
            \State $\sigma_{k+1} \in [0,\sigma_\ell^k]$
        \ElsIf{$\eta_1 \leq \rho_k \leq \eta_2$}
            \State $\sigma_{k+1} \in [\sigma_\ell^k, \delta_1 \sigma_\ell^k]$
        \EndIf
        \State $B_{k+1} \leftarrow \mathrm{UpdateHessianApproximation()}$
        \EndFor
    \end{algorithmic}
\end{breakablealgorithm}
\bigskip

\begin{proposition} \label{prop:acr-example}
    Let $\delta_1, \delta_2 \in \mathbb{R}$ such that $\delta_2 \geq \delta_1 > 1$, $\eta_1, \eta_2 \in \mathbb{R}$ such that $1 > \eta_2 \geq \eta_1 > 0$, $\sigma_0 > 0$, and $J \in \mathbb{N}$. Define $\delta = \frac{1}{\delta_2^{1/2}}$ and $S_0 = 0$; for $j = 1,\ldots, J$, let $S_j = S_{j-1} + \delta^{j - 2^{j-1}}$; and, for $j = 0,\ldots,J$ 
    \begin{equation}
        ( f_j, f_j', f_j'' ) = 
        \begin{cases}
            \left(0, -\sigma_0, 0  \right) & j = 0\\
            \left( -\frac{2(\eta_2+1)\sigma_0}{3} \sum_{k=0}^{j-1} \left( \delta^{k+2 - 2^{k+1}} \right)^{3}, -\sigma_0\left( \delta^{j+2 - 2^{j+1}} \right)^2, 0 \right) & j = 1,\ldots,J
        \end{cases}.
    \end{equation}
    Let $\Delta = \frac{1-\delta}{2}$, and $F_J: \mathbb{R} \to \mathbb{R}$ be defined as in \cref{eqn-ee-objective}, and define $B_j = \ddot F(\theta_j)$ (i.e., use the true Hessian). Let $\theta_0 = 0$ and let $\lbrace \theta_1,\ldots,\theta_J \rbrace$ be generated by \cref{alg:acr}. Then, the total number of objective function evaluations taken by \cref{alg:acr} upon acceptance of iterate $\theta_j$ is $2^{j}$ for $j=1,\ldots,J$.
\end{proposition}
\begin{proof}
    First, the function $F_J$ has all the properties in \cref{prop-ee-properties}, as our components satisfy all assumptions. Now, before proceeding to the main proof, we provide a global minimizer of $U_j(\gamma; \sigma)$ for any $\sigma > 0$, in the special case that $\theta_j$ satisfies $B_j = \ddot F(\theta_j) = 0$ and $\dot F(\theta_j) \not= 0$. Checking first and second order sufficient conditions, we can rewrite $U_j$, and derive $\dot U_j, \ddot U_j$ as 
    \begin{equation}
        \begin{aligned}
            U_j(\gamma; \sigma) = F(\theta_j) + \gamma^\intercal \dot F(\theta_j) +& 
            \begin{cases}
                \frac{\sigma}{3}s^3 & s \geq 0 \\
                -\frac{\sigma}{3}s^3 & s < 0
            \end{cases},
            \text{\quad}
            \dot U_j(\gamma; \sigma) = \dot F(\theta_j) + 
            \begin{cases}
                \sigma s^2 & s \geq 0 \\
                -\sigma s^2 & s < 0
            \end{cases}\\
            &\ddot U_k(\gamma; \sigma) = 
            \begin{cases}
                2\sigma s & s \geq 0\\
                -2\sigma s & s < 0
            \end{cases}
        \end{aligned}
    \end{equation}
    One can verify that two possible minimizers are $\sqrt{-\dot F(\theta_j)/\sigma}, -\sqrt{\dot F(\theta_j)/\sigma}$, only one of which is in $\mathbb{R}$ depending on the sign of the gradient. We now proceed to prove the statement in the proposition by induction. By construction, $\theta_0 = S_0$, so that $B_0 = \ddot F(S_0) = 0$ and $\dot F(S_0) = -\sigma_0 < 0$. To compute $\gamma^0_0$, we globally minimize $U_k(\gamma; \sigma^0_0)$, which by above yields $\gamma^0_0 = \sqrt{-\dot F(\theta_0)/\sigma^0_0} = 1$. Therefore, since $\sigma^0_0 = \sigma_0$
    \begin{equation}
        F(\theta_0) - F(\theta_0 + \gamma_0) = -F(S_1) = \frac{2(\eta_2+1)\sigma_0}{3} \text{,\quad} F(\theta_0) - U_0(\gamma^0_0; \sigma_0) = \frac{2\sigma_0}{3}.
    \end{equation}
    This implies that $\rho^0_0 = (\eta_2+1) > \eta_2$, and we accept $\theta_1 = S_1$ using two objective function evaluation ($F(\theta_0)$ and $F(S_1)$). Additionally, since $\rho^0_0 > \eta_2$ we can set $\sigma_1 = \sigma^0_0 = \sigma_0$. If $J = 1$ we are done, otherwise suppose for $j \in \{1,...,J-1\}$, $\theta_j = S_j$, $\sigma_j = \sigma_0$, and the algorithm has taken a total of $2^j$ objective evaluations. We generalize to the case of $j+1$ by showing that
    \begin{equation}
        \theta_j + \delta^{2^j-1} \frac{1}{\sqrt{\sigma_0}} \sqrt{-\dot F(\theta_j)} = \theta_j + \delta^{2^j-1} \delta^{j+2-2^{j+1}} = S_j + \delta^{j+1-2^j} = S_{j+1}
    \end{equation}
    is the accepted iterate. To do this, we prove that for all $\ell \in \{0,1,...,2^j-2\}$, the trial step $\gamma_\ell^j$ results in a ratio $\rho_{\ell}^j < \eta_1$. First, for any $\ell \in \{0,1,...,2^j-2\}$, since $\dot F(\theta_j) = \dot F(S_j) < 0$ and $\ddot F(\theta_j) = \ddot F(S_j) = 0$ the solution of $\arg\min U_j(\gamma; \delta_2^\ell \sigma_0)$ is $\gamma_\ell^j = \sqrt{\frac{-\dot F(\theta_j)}{\delta_2^\ell \sigma_0}} = \delta^\ell \sqrt{\frac{-\dot F(\theta_j)}{\sigma_0}}$ using the identity $\delta = \frac{1}{\sqrt{\delta_2}}$. We now proceed by induction. For the base case, suppose $\ell = 0$, then our penalty parameter is $\sigma_0^j = \sigma_0$ by the inductive hypothesis. Therefore, the trial step is $\gamma_0^j = \sqrt{\frac{-\dot F(\theta_j)}{\sigma_0}}$ and using the fact that $\delta^{2^j-2} \leq 1$, we obtain that $\theta_j + \gamma_0^j$ satisfies
    \begin{equation}
        S_{j+1} + \frac{1-\delta}{2} \leq S_j + \delta^{2^j-2}\delta^{j+2-2^{j+1}} \leq \theta_j + \sqrt{\frac{-\dot F(\theta_j)}{\sigma_0}} \leq S_j + \delta^{j+2-2^{j+1}} \leq S_{j+2} - \frac{1-\delta}{2}.
    \end{equation}
    This implies that $F(\theta_j + \gamma_0^j)  = 0$. Since $F(S_j) < 0$ this implies $\rho_0^j < 0$, so the algorithm rejects the iterate and inflates the penalty parameter. Since $\sigma_1^j$ can be any value within the interval $[\delta_1\sigma_0^j, \delta_2\sigma_0^j]$, let $\sigma_{1}^j = \delta_2\sigma_0^j = \delta_2\sigma_0$. The inductive step is handled in exactly the same way, therefore all trial iterates $\theta_j + \gamma_\ell^j$ are rejected for $\ell \in \{0,1,...,2^j-2\}$. Finally, the iterate $S_{j+1}$ is accepted, because after $\ell = 2^j-2$ the penalty parameter is increased to $\delta_2^{2^j-1} \sigma_0$, so by letting $\gamma_{2^j-1}^j =  \delta^{2^j-1} \sqrt{\frac{-\dot F(\theta_j)}{\sigma_0}}$ we obtain
    \begin{equation}
        \frac{F(S_j) - F(S_{j+1})}{F(S_j) - U_j\left( \gamma_{2^j-1}^j; \delta^{2^j-1}\sigma_0 \right)} = \frac{\frac{2(\eta_2+1)\sigma_0}{3} \left( \delta^{j+2 - 2^{j+1}} \right)^{3} }{ \delta^{2^j-1} \frac{2}{3\sqrt{\sigma_0}} \sigma_0^{3/2} \left( \delta^{j+2 - 2^{j+1}} \right)^{3} } = \frac{\eta_2+1}{\delta^{2^j-1}} > \eta_2
    \end{equation}
    Since $\rho_{2^j-1}^j > \eta_2$, we can reset $\sigma_{j+1} = \sigma_0$. To finish the original inductive proof, we have just shown to accept $\theta_{j+1}$, \cref{alg:acr} requires $2^j$ objective function evaluations (excluding $F(\theta_j)$ as it was calculated the previous iteration). Therefore, in total, the algorithm has taken $2^{j+1}$ function evaluations upon acceptance of $\theta_{j+1}$, as by the inductive hypothesis the algorithm has taken $2^j$ to accept $\theta_j$.
\end{proof}

\subsection{Lipschitz Constant Line Search Methods} \label{subsec:lipschitz-line-search}
Lastly, we consider Lipschitz Constant Line Search Methods. An example of such a procedure is shown in \cref{alg:dynamic_method} which was studied in \cite{curtis2018exploiting}. This class of methods usually utilize conservative upper bound models for the objective using Lipschitz constants, and try to find the smallest such constant that satisfies the model through line search. The step is then based off this estimated parameter. Another method in this family was introduced in \cite{nesterov2012GradientMF}, and can be handled similarly.
\bigskip
\begin{breakablealgorithm}
    \caption{Dynamic Method, \cite[see][Algorithm 2]{curtis2018exploiting}}
    \label{alg:dynamic_method}
    \begin{algorithmic}[1]
        \Require $\delta_1 \in (1,\infty)$, $L_0 \in (0, \infty), \sigma_0 \in (0, \infty)$, $\theta_0 \in \mathbb{R}^n$
        \Require $\mathrm{NegativeCurvatureDirection()}$
        \For{$k = 0,1,2,...$}
            \If{$\ddot F(\theta_k) \succeq 0$}
                \State $s'_k \leftarrow 0$
            \Else
                \State $s'_k \leftarrow \mathrm{NegativeCurvatureDirection()}$ \Comment{We allow any negative curvature direction}
            \EndIf
            \State $s_k \leftarrow -\dot F(\theta_k)$
            \State $m_k(L) \leftarrow -\dot F(\theta_k)^\intercal s_k / (L \inlinenorm{s_k}_2^2)$ \Comment{Gradient step size}
            \State $c_k \leftarrow (s'_k)^\intercal \ddot F(\theta_k)s'_k$
            \State $m'_k(\sigma) \leftarrow \left( -c_k + \sqrt{c_k^2-2\sigma\inlinenorm{s'_k}_2^3 \dot F(\theta_k)^\intercal s'_k}  \right)/(\sigma \inlinenorm{s'_k}_2^3)$ \Comment{Negative curvature step size; $0$ if $s_k' = 0$}
            
            \State $U_k(m; L) \leftarrow m\dot F(\theta_k)^\intercal s_k + \frac{1}{2} L m^2 \inlinenorm{s_k}_2^2$
            \State $U_k'(m'; \sigma) \leftarrow m'\dot F(\theta_k)^\intercal s'_k + \frac{1}{2} (m')^2 (s'_k)^\intercal \ddot F(\theta_k) s_k' + \frac{\sigma}{6}(m')^3 \inlinenorm{s'_k}_2^3$
            \State $\ell, L_\ell^k, \sigma_\ell^k \leftarrow 0, L_k, \sigma_k$
            \While{ \text{true} }
                \If{ $U_k(m_k(L_\ell^k); L_\ell^k) \leq U_k'(m'_k(\sigma_\ell^k); \sigma_\ell^k)$ }
                    \If{$F(\theta_k + m_k(L_\ell^k)s_k) \leq F(\theta_k) + U_k(m_k(L_\ell^k); L_\ell^k)$}
                        \State $\gamma_k \leftarrow m_k(L_\ell^k)s_k$
                        \State $\textbf{exit loop}$
                    \Else
                        \State $L_{\ell+1}^k \leftarrow \delta_1 L_\ell^k$
                    \EndIf
                \Else
                    \If{$F(\theta_k + m'_k(\sigma_\ell^k)s'_k) \leq F(\theta_k) + U_k'(m'_k(\sigma_\ell^k); \sigma_\ell^k)$}
                        \State $\gamma_k \leftarrow m'_k(\sigma_\ell^k)s'_k$
                        \State $\textbf{exit loop}$
                    \Else
                        \State $\sigma_{\ell+1}^k \leftarrow \delta_1 \sigma_\ell^k$
                    \EndIf
                \EndIf
            \State $\ell \leftarrow \ell + 1$
            \EndWhile
            \State $\theta_{k+1} \leftarrow \theta_{k} + \gamma_k$
            \State $L_{k+1} \in (0, L_\ell^k]$
            \State $\sigma_{k+1} \in (0, \sigma_\ell^k]$
        \EndFor
    \end{algorithmic}
\end{breakablealgorithm}
\bigskip

We now provide a class of objective functions such that there is an exponential increase in objective function evaluations between accepted iterates up to the $J$th acceptance.

\begin{proposition}
    Let $\delta_1 \in (1,\infty)$, $L_0 \in (0, \infty), \sigma_0 \in (0, \infty)$, and $J \in \mathbb{N}$. Set $\delta = \frac{1}{\delta_1}$, and $S_0 = 0$; for $j = 1,...,J$, let $S_j = S_{j-1} + \delta^{j-2^{j-1}}$; and, for $j = 0,\ldots,J$
    \begin{equation}
        (f_j, f_j', f_j'') = 
        \begin{cases}
            \left( 0, -L_0, d_0' \right) & j = 0 \\
            \left(-\frac{L_0}{2} \sum_{k = 0}^{j-1} \delta^{2k + 4 - 2^{k+2}}, -L_0\delta^{j+2-2^{j+1}}, d_j' \right) & j = 1,\ldots, J
        \end{cases},
    \end{equation}
    where $d_j' \in \mathbb{R}_{\ge 0}$ is any non-negative real number. Let $\Delta = (1-\delta)/2$, and $F_J : \mathbb{R}\to\mathbb{R}$ be defined as in \cref{eqn-ee-objective}. Let $\theta_0 = 0$ and let $\{\theta_1,...,\theta_J\}$ be generated by \cref{alg:dynamic_method}. Then, the total number of objective function evaluations taken by \cref{alg:dynamic_method} upon acceptance of iterate $\theta_j$ is $2^j$ for $j = 1,...,J$.
\end{proposition}
\begin{proof}
    First the function $F_J$ has all the properties in \cref{prop-ee-properties}, as our components satisfy all assumptions. We now proceed by induction. By construction, $\theta_0 = S_0$ and $\ddot F(\theta_0) = \ddot F(S_0) \geq 0$, therefore $s_0' = 0$. The algorithm also sets $L_0^0 = L_0$, so by our choice of $s_0$, $m_0(L_0^0) = 1/L_0$ and $U_0(m_0(L^0_0); L^0_0) = -\dot F(\theta_0)^2 / (2L_0)$. Since $U_0(m_0(L^0_0);L^0_0) < 0 = U'_0(m_0'(\sigma^0_0);\sigma^0_0)$, the algorithm checks to see if the first order upper bound model is satisfied. Indeed,
    \begin{equation}
        F(\theta_0 - (1/L_0)\dot F(\theta_0) ) = F(S_1) = -\frac{L_0}{2} = F(S_0) - \frac{1}{2L_0} \dot F(S_0)^2 = F(S_0) + U_0(m_0(L_0^0); L_0^0).
    \end{equation}
    Therefore, $S_1$ is accepted with two objective function evaluations ($F(S_1)$ and $F(\theta_0)$). Furthermore, we can set $L_1 = L_0 \in (0,L^0_0]$. If $J = 1$ we are done, otherwise suppose that for $j \in \{1,...,J-1\}$ that $\theta_j = S_j$, $L_j = L_0$, and $2^j$ objective functions have been taken to accept $\theta_j$. We generalize to the case of $j+1$ by showing that
    \begin{equation}
        \theta_j - \frac{1}{\delta_1^{2^j-1} L_0}\dot F(\theta_j) = S_j + \delta^{2^j-1} \delta^{j+2-2^{j+1}} = S_{j+1}
    \end{equation}
    is the accepted iterate. We do this by showing that all iterations $\ell \in \{0,...,2^j-2\}$ result in a rejection, and only $L_j$ is inflated. 
    
    To start, from the inductive hypothesis we have $\ddot F(\theta_j) \geq 0$, therefore $s_j' = 0$. By our choice of $s_j$, for any $L \in \mathbb{R}$, $m_j(L) = 1/L$. This implies that for all $\ell \in \{0,...,2^j-2\}$, $U_j(m_j(\delta^\ell_1 L_0); \delta^\ell_1 L_0) = - (1/2\delta^\ell_1 L_0) \dot F(\theta_j)^2$; thus, for any $\sigma > 0$, $U_j(m_j(\delta^\ell_1 L_0); \delta^\ell_1 L_0) < 0 = U_j'(m_j'(\sigma); \sigma)$. We now show by induction that for all $\ell \in \{0,..., 2^j-2\}$, the trial iterate on iteration $\ell$ is indeed $\theta_j + m_k(\delta^\ell_1 L_0)s_k$, and it is rejected. For the base case, $\ell = 0$ and $L_0^j = L_0$, which by $U_j(m_j(L_0); L_0) < U_j'(m_j'(\sigma_j); \sigma_j)$, means the algorithm checks the first order upper bound model. Note, that the proposed iterate $\theta_j - (1/L_0^j) \dot F(\theta_j) = \theta_j - (1/L_0) \dot F(\theta_j)$ satisfies
    \begin{equation}
        S_{j+1} + \frac{1-\delta}{2} \leq S_j + \delta^{2^j-2}\delta^{j+2-2^j} \leq \theta_j - (1/L_0)\dot F(\theta_j) \leq S_j + \delta^{j+2-2^j} = S_{j+2}-\delta^{j+1-2^j} \leq S_{j+2}-\frac{1-\delta}{2}.
    \end{equation}
    This implies that $F(\theta_j - (1/L_0)\dot F(\theta_j)) = 0$, yet by the inductive hypothesis $F(\theta_j) + U_j(m_j(L_0); L_0) < 0$. Therefore, the proposed iterate is rejected, and $L_1^j = \delta_1 L_0^j = \delta_1 L_0$. The inductive step is handled in exactly the same way. 
    
    Finally, we show that for $\ell = 2^j-1$ we accept the proposed iterate. By the above inductive proof, at iteration $\ell = 2^j-1$, $L_{2^j-1}^j = \delta_1^{2^j-1}L_0$. Recall, $\theta_j - (1/\delta_1^{2^j-1} L_0)\dot F(\theta_j) = S_{j+1}$, and since $\delta^{2j+4-2^{j+2}} \geq \delta^{2^j-1}\delta^{2j+4-2^{j+2}}$ for all $j \in \mathbb{N}$, by adding the extra terms from $F(S_{j+1})$ and multiplying by $-L_0/2$ we obtain
    \begin{equation}
        \begin{aligned}
            F(S_{j+1}) = -\frac{L_1}{2} \sum_{k = 0}^{j} \delta^{2k + 4 - 2^{k+2}} &\leq -\frac{L_1}{2} \sum_{k = 0}^{j-1} \delta^{2k + 4 - 2^{k+2}} - \frac{L_1}{2}\delta^{2^j-1}\delta^{2j+4-2^{j+2}}\\
            &= F(S_j) - \frac{\delta^{2^j-1}}{2L_0} \dot F(S_j)^2 = F(S_j) + U_j(m_j(L_{2^j-1}^j); L_{2^j-1}^j).
        \end{aligned}
    \end{equation}
    Therefore, $\theta_{j+1} = S_{j+1}$ is the accepted iterate. Additionally, since $\delta_1  > 1$, we can set $L_{j+1} = L_0 \in (0, \delta_1^{2^j-1}L_0]$. Finally, the algorithm took $2^j$ evaluations to accept this point (excluding $F(\theta_j)$ as this was calculated at the previous point). By the inductive hypothesis, $2^j$ have already been taken, implying a total of $2^{j+1}$ objective function evaluations upon acceptance of $\theta_{j+1}$.
\end{proof}

\section{Conclusion} \label{sec:conclusion}
Gradient methods have experienced a growth in theoretical and methodological developments to meet the demands of data science applications. 
To guarantee convergence and produce complexity results, often restrictive smoothness conditions are assumed (e.g., global Lipschitz continuity of the gradient).
We have shown, however, that many data science applications do not satisfy global Lipschitz continuity of the gradient function or other recently developed concepts of smoothness in literature, but rather are better modeled as having locally Lipschitz continuous gradient functions.
Furthermore, we illustrate that, under local Lipschitz continuity of the gradient function, 
optimization algorithms that do not make use of the objective can have optimality gap divergence, 
and optimization algorithms that make use of the objective can require an exponentially increasing number of objective evaluations.
This illustrates that current methods for optimization in data science either face practical challenges under general smoothness conditions that more accurately reflect optimization problems arising in applications.
This motivates the need for new methods that are both reliable and that can mitigate growth in complexity.

\bibliographystyle{siamplain}
\bibliography{references}

\end{document}